\documentclass[english, 10pt,reqno]{amsart}
\usepackage{geometry}            
\usepackage{amssymb,amsmath,amsthm,amsfonts,color}
\usepackage{mathrsfs,dsfont, comment,mathscinet}
\usepackage{graphicx}
\usepackage{epstopdf}
\usepackage{mathtools}
\usepackage{babel}
\usepackage{enumerate,esint}
\usepackage{caption}
\usepackage{subcaption}

\usepackage{indentfirst}
\usepackage{bm}
\usepackage{picinpar}
\usepackage{lipsum}

\usepackage[colorlinks=true, pdfstartview=FitV, linkcolor=blue, citecolor=blue, urlcolor=blue]{hyperref}

\usepackage[toc,page]{appendix}

\usepackage{algorithm2e}
\usepackage{etoolbox}
\usepackage{authblk}
\usepackage{amsaddr}
\usepackage{float}

\allowdisplaybreaks 

\mathtoolsset{showonlyrefs} 

\makeatletter

\usepackage{fancyhdr}
 
\makeatletter
\patchcmd{\@maketitle}
  {\ifx\@empty\@dedicatory}
  {\ifx\@empty\@date \else {\vskip3ex \centering\footnotesize\@date\par\vskip1ex}\fi
   \ifx\@empty\@dedicatory}
  {}{}
\patchcmd{\@maketitle}
  {\ifx\@empty\@date\else \@footnotetext{\@setdate}\fi}
  {}{}{}
\makeatother

\usepackage{lineno}
\newcommand*\patchAmsMathEnvironmentForLineno[1]{%
  \expandafter\let\csname old#1\expandafter\endcsname\csname #1\endcsname
  \expandafter\let\csname oldend#1\expandafter\endcsname\csname end#1\endcsname
  \renewenvironment{#1}%
     {\linenomath\csname old#1\endcsname}%
     {\csname oldend#1\endcsname\endlinenomath}}%
\newcommand*\patchBothAmsMathEnvironmentsForLineno[1]{%
  \patchAmsMathEnvironmentForLineno{#1}%
  \patchAmsMathEnvironmentForLineno{#1*}}%
\AtBeginDocument{%
\patchBothAmsMathEnvironmentsForLineno{equation}%
\patchBothAmsMathEnvironmentsForLineno{align}%
\patchBothAmsMathEnvironmentsForLineno{flalign}%
\patchBothAmsMathEnvironmentsForLineno{alignat}%
\patchBothAmsMathEnvironmentsForLineno{gather}%
\patchBothAmsMathEnvironmentsForLineno{multline}%
}

\begingroup
\newtheorem{theorem}{Theorem}[section]

\newtheorem{corollary}[theorem]{Corollary}

\endgroup

\theoremstyle{definition}
\begingroup
\newtheorem{define}[theorem]{Definition}
\newtheorem{remark}[theorem]{Remark}

\newtheorem{proposition}[theorem]{Proposition}
\newtheorem{notation}[theorem]{Notation}
\newtheorem{assumption}[theorem]{Assumption}

\endgroup

\newcommand\ba[1]{\begin{align}\label{#1}}
\newcommand\ea{\end{align}}
\newcommand\bas{\begin{align*}}
\newcommand\eas{\end{align*}}

\newcommand\ee{\end{equation}}
\newcommand\be{\begin{equation}}
\newcommand\ees{\end{equation*}}
\newcommand\bes{\begin{equation*}}

\mathchardef\emptyset="001F

\newcommand{\e}{\varepsilon}

\newcommand{\R}{{\mathbb R}}

\newcommand{\rn}{{{\R}^N}}
\newcommand{\rk}{{{\R}^K}}

\newcommand{\wto}{\rightharpoonup}
\newcommand{\wtos}{\mathrel{\mathop{\rightharpoonup}\limits^*}}

\newcommand{\CC}{{\mathcal A}}
\newcommand{\N}{{\mathbb{N}}}

\newcommand\norm[1]{\left\|#1\right\|}

\newcommand{\abs}[1]{\left\lvert#1\right\rvert} 
\newcommand{\fsp}[1]{\left(#1\right)} 
\newcommand{\fmp}[1]{\left[#1\right]}
\newcommand{\flp}[1]{\left\{#1\right\}}  
\newcommand{\fjp}[1]{\left<#1\right>}

\newcommand{\vp}{\varphi}

\newcommand{\limn}{\lim_{n\rightarrow\infty}}

\newcommand{\lime}{\lim_{\e\rightarrow0}}

\newcommand{\B}{{\mathscr B}}

\newcommand{\loc}{{\operatorname{loc}}}

\newcommand{\seqn}[1]{\left\{#1\right\}_{n=1}^\infty}  
\newcommand{\seqk}[1]{\left\{#1\right\}_{k=1}^\infty}  
\newcommand{\seql}[1]{\left\{#1\right\}_{l=1}^\infty}  
\newcommand{\seqe}[1]{\left\{#1\right\}_{\e>0}}

\newcommand{\T}{{\mathbb T}}
\newcommand{\FT}{{\mathcal T}}

\definecolor{CMUred}{RGB}{153,0,0}
\definecolor{CMUgreen}{RGB}{0,135,81}
\definecolor{CMUblue}{RGB}{0,51,127}
\definecolor{Pblue}{RGB}{87,158,208}

\newcommand{\argmin}{{\operatorname{arg\,min}}}

\newcommand{\mb}{{\mathcal{M}_b}}
\newcommand{\liminfn}{{\liminf_{n\to\infty}}}

\def\argmin{\mathop{\rm arg\, min}}

\numberwithin{equation}{section}

\newcommand{\normmm}[1]{{\left\vert\kern-0.25ex\left\vert\kern-0.25ex\left\vert #1 
    \right\vert\kern-0.25ex\right\vert\kern-0.25ex\right\vert}}

\title{Adaptive image processing: a bilevel structure learning approach for mixed-order total variation regularizers}

\author[P. Liu] {Pan Liu}
 \address[Pan Liu]{Centre of Mathematical Imaging and Healthcare,\\ 
Department of Pure Mathematics and Mathematical Statistics, \\
 University of Cambridge, Wilberforce Road, Cambridge CB3 0WA, UK}
 \email[P. Liu] {panliu.0923@maths.cam.ac.uk}

\subjclass[2010]{26B30, 94A08, 	47J20}
\keywords{image processing, optimal training scheme, higher order differential operators, $\Gamma$-convergence}

\date{\today}                                           

\begin{document}
\begin{abstract}
A class of mixed-order \emph{PDE}-constraint regularizer for image processing problem is proposed, generalizing the standard first order total variation $(TV)$. A semi-supervised (bilevel) training scheme, which provides a simultaneous optimization with respect to parameters and new class of regularizers, is studied. Also, A finite approximation method, which used to solve the global optimization solutions of such training scheme, is introduced and analyzed.
\end{abstract}
\maketitle
\tableofcontents

\thispagestyle{empty}

\section{Introduction}\label{sec_introduction}
The use of variational technics with non-smooth regularizers in image processing has become popular in the last decades. One of the most successful approaches is introduced in the celebrated work \cite{rudin1992nonlinear} which relies on the so called \emph{ROF} total-variational functional
\be\label{tv_initial_denoisy}
\mathcal I(u):=\norm{u-u_\eta}_{L^2(Q)}^2+\alpha TV(u),
\ee
where $u_\eta\in L^2(Q)$ is a given corrupted image, $Q:=(0,1)^2$ represents the unit square, $\alpha\in\R^+$ is an \emph{intensity parameter}, and $TV(u)$ stands for the total variation of $u$ in $Q$ (see \cite{evans2015measure}). In the simple case that $u\in W^{1,1}(Q)$, we have
\be\label{basic_tv_intro}
TV(u) = \int_Q\abs{\nabla u}\,dx = \int_Q\fsp{\abs{\partial^1_1u(x)}^2+\abs{\partial^1_2 u(x)}^2}^{1/2}dx
\ee
One advantage of using the $TV$ regularization is it promotes piecewise constant reconstructions, thus preserving edges. 
However, this also leads to blocky-like artifacts in the reconstructed image, an effect known as stair-casing. To mitigate this effect, and also to explore possible improvements, the following methods has been introduced and studied:
\begin{enumerate}[1.]
\item
 using higher-order extensions (\cite{bredies2010total, liu2016Image});
 \item
 changing the underlying Euclidean norm (\cite{carolapani2018bilevel});
 \item
  introducing fractional order derivatives \cite{2018arXiv180506761L,liu2016weightedreg}.  
\end{enumerate}
These methods introduces collections of regularizers which generalizes $TV$ seminorm.  For example, in \cite{carolapani2018bilevel}, the underlying Euclidean norm of $TV$ seminorm is generalized from $p=2$, used in \eqref{basic_tv_intro}, to $p\in [1,+\infty]$ by
\be
TV_p(u) = \int_Q\abs{\nabla u}_p\,dx = \int_Q\fsp{\abs{\partial^1_1u(x)}^p+\abs{\partial^1_2 u(x)}^p}^{1/p}dx.
\ee
In \cite{2018arXiv180506761L}, the order of derivative is generalized from $r=1$, used in \eqref{basic_tv_intro}, to $r\in\R^+$, by
\be
TV^r(u) = \int_Q\abs{\nabla^r u}\,dx = \int_Q\fsp{\abs{\partial^r_1u(x)}^2+\abs{\partial^r_2 u(x)}^2}^{1/2}dx,
\ee
in which the fractional order derivative is realized by using the \emph{Riemann-Liouville} fractional order derivative (see \cite{MR1347689} for definition). In both \cite{2018arXiv180506761L,liu2016weightedreg}, it has been shown that for given corrupted image $u_\eta$, a carefully selected \emph{regularizer} parameter $p\in[1,+\infty]$ (resp. $r\in\R^+$) allows $TV_p$ (resp. $TV^r$) to provide improved imaging processing result, and such selection can be done automatically by using a bi-level training scheme which will be detailed below.  \\\\
In general, with a reliable selection mechanism, the imaging processing results would certainly be improved if we could further expand the collections of regularizers. To this purpose, in this paper we introduce a family of novel $TV$-like \emph{PDE}-constraint regularizer (semi-norm), say $PV_\B$, by
\be\label{abextension}
PV_\B(u):=\abs{\B u}_{\mb(Q;\rk)}.
\ee
Here $\abs{\cdot}_\mb$ denotes the \emph{Radon} norm of a measure, and $\B$: $L^1(Q)\to \mathcal D'(Q,\rk)$ is a linear differential operator (see Notation \ref{operator_B}). 
In the simple case $\B=\nabla u$, we recover the total variation $TV$ seminorm. We remark that the abstract framework studied in \eqref{abextension} naturally incorporates the recent \emph{PDE}-based approach to image denoising problems formulated  in \cite{barbu.marinoschi}, and also allows us to simultaneously describe a variety of different image-processing techniques.\\\\
The aim of this paper is threefold. First, we provide a rigorous and detailed analysis of the properties of the $PV_\B$ seminorm, such as the approximation by smooth functions, 
lower semi-continuity with respect to both function $u$ and operator $\B$, and a point-wise characterization of the sub-gradient of $PV_\B$.\\\\ 
 The second result is the study of the aforementioned selection mechanism, realized by a semi-supervised (bilevel) training scheme 
 defined in machine learning (see \cite{chen2013revisiting, chen2014insights, domke2012generic,tappen2007learning,de2013image,kunisch2013bilevel}). 
 For example, we could apply the bilevel training scheme to determine the optimal value of $\alpha\in\R^+$ from \eqref{tv_initial_denoisy}, which controls the strength of the regularizer. 
 More precisely, we assume that the corrupted image $u_\eta$ can be decomposed as $u_\eta=u_c+\eta$ where $u_c\in L^2(Q)$ represents a noise-free clean image (the perfect data), and $\eta$ encodes noise, and we call $(u_\eta,u_c)$ as \emph{training set}. Then, a bilevel training scheme, say Scheme $\mathcal B$, for determining the optimal intensity parameter $\alpha$ can be formulated as follows: 
\begin{flalign}
\text{Level 1. }&\,\,\,\,\,\,\,\,\,\,\,\,\,\,\,\,\,\,\,\,\,\, \alpha_\T\in\mathbb A[\T]:=\argmin\flp{\norm{u_\alpha-u_c}_{L^2}^2:\,\, \alpha\in \T},\tag{{$\mathcal B$-L1}}\label{scheme_B1_BV}&\\
\text{Level 2. }&\,\,\,\,\,\,\,\,\,\,\,\,\,\,\,\,\,\,\,\,\,\,u_\alpha:=\argmin\flp{\norm{u-u_\eta}_{L^2(Q)}^2+\alpha {TV_2(u)}:\,\, u\in BV(Q)}\tag{{$\mathcal B$-L2}}\label{scheme_B2_BV},&
\end{flalign}
where $\T:=\operatorname{cl}(\R^+)$, used in \eqref{scheme_B1_BV}, is called the \emph{training ground}. Roughly speaking, Level 1 problem in \eqref{scheme_B1_BV} looks for an $\alpha$ that solves the minimum $L^2$-distance to the clean image $u_c$, subject to the minimizing problem \eqref{scheme_B2_BV}. That is, scheme $\mathcal B$ is able to optimally adapt itself to the given ``perfect data" $u_c$.\\\\
In the same spirit, in order to identify the optimal operator $\B$ in $PV_\B$ for a given training set $(u_\eta,u_c)$, we introduce the scheme $\mathcal T$ (\eqref{training_level_bi}-\eqref{result_level_bi}) defined as
 \begin{flalign}
{\text{Level 1. }}&\,\,\,\,\,\,\,\,\,\,\,\,\,\,\,\,\,\,\,(\alpha_\T,\B_\T)\in\mathbb A[\T]:= \argmin\flp{\norm{u_c-u_{\alpha,\B}}_{L^2(Q)}^2:\,\,(\alpha,\B)\in\T},\tag{$\mathcal T$-L1}\label{training_level_bi} &\\
{\text{Level 2. }}&\,\,\,\,\,\,\,\,\,\,\,\,\,\,\,\,\,\,\,u_{\alpha,\B}:=\argmin\flp{\norm{u-u_\eta}_{L^2(Q)}^2+ \alpha PV_\B(u),\,\, u\in L^1(Q)}\tag{$\mathcal T$-L2}\label{result_level_bi}.&
\end{flalign}
In \eqref{training_level_bi}, we expand the training ground to $\T:=\operatorname{cl}(\R^+)\times \Sigma$ to incorporate the new parameter $\B\in\Sigma$, 
where $\Sigma$ denotes a closed collection of operators $\B$ (see Notation \ref{operator_B}, Notation \ref{def_training_set}, and \eqref{kappa_training_ground} for details).
We remark that the expanded training ground $\T$ allows scheme $\mathcal T$ to optimize the regularizer $PV_{\B}(u)$ and intensity parameter $\alpha$ simultaneously. 
We summarize the main result in the following theorem.
\begin{theorem}[see Theorem \ref{main_thm}]\label{main_thm_intro}
The training scheme $\mathcal T$ admits at least one solution $(\alpha_\T, \B_\T)\in\mathbb T$, and provides an associated optimally reconstructed image $u_{\alpha_\T, \B_\T}\in BV_{ \B_\T}(Q)$.
\end{theorem}
In the third part of this article we focus on how to numerically determine the optimal solution of scheme $\mathcal T$, or equivalently, compute global minimizers of 
the \emph{assessment function} $\mathcal A(\alpha,\B)$: $\T\to\R^+$ defined as 
\be\label{cost_map_intro}
\mathcal A(\alpha,\B):=\norm{u_{\alpha,\B}-u_c}_{L^2(Q)}^2,
\ee
where $u_{\alpha,\B}$ is obtained from \eqref{result_level_bi}. However, as shown in \cite{carolapani2018bilevel} that even in the simplest case with
$\B=\nabla$ (i.e. $PV_\B=TV$), the assessment function $\mathcal A(\alpha,\nabla)$ is not quasi-convex (in the sense of \cite{MR1819784}, or simply convex), 
and hence the methods such as \emph{Newton's descent} or \emph{Line search} might get trapped in a local minimum. To overcome this difficulty, 
we introduce the concept of the \emph{acceptable optimal solution}. To be precise, we say the solution $(\alpha',\B')$ is an acceptable optimal solution of scheme $\mathcal T$ with the given error $\e>0$ if 
\be\label{accptable_train_result_intro}
\abs{\CC( \alpha_\T, \B_\T)-\CC(\alpha',\B')}<\e,
\ee
where $(\alpha_\T,\B_\T)$ is a global minimum obtained from \eqref{training_level_bi}.\\\\
To compute such acceptable optimal solution，
we use a finite approximation method, originally introduced and studied in \cite{carolapani2018bilevel}, and generalized in Section \ref{sec_numerical_finiteapprox} 
to fit our new regularizer $PV$. To this aim, and also for the numerical realization of scheme $\FT$, we add the following \emph{box-constraint} on the training ground $\T$.
\begin{itemize}
\item
The intensity parameter $\alpha$ is contained in a closet interval $[0,P]$, where the box-constraint constant $P>0$ can be chosen by the user;
\item
the collection $\Sigma$ of operator $\B$ satisfies an additional continuity assumptions, such as, for any $\B_1$, $\B_2\in \Sigma$,
\be
\abs{PV_{\B_1}(u)-PV_{\B_2}(u)}\leq O\fsp{\abs{\B_1-\B_2}}\min\flp{PV_{\B_1}(u),PV_{\B_2}(u)},
\ee
where $O(\cdot)$ denotes the big-$O$ notation.
\end{itemize}

Then, the finite approximation method is constructed based on a sequence of (finite) training sets $\T_l$, indexed by $l\in\N$, in which (where $\mathcal H^0(\cdot)$ denotes the counting measure)
\be
\mathcal H^0\fsp{\T_l}<+\infty\text{ and }\T\subset \operatorname{cl}\fsp{\bigcup_{l\in\N}\T_l}.
\ee
For the precise definition of $\T_l$ we refer readers to Definition \ref{def_BCFTG}. We remark that, since $\mathcal H^0\fsp{\T_l}<+\infty$ for each $l\in \N$ fixed, we could evaluate $\CC(\alpha,\B)$ at each element of $\T_l$ and determine the optimal solution(s)
\be
(\alpha_{\T_l},\B_{\T_l})\in\mathbb A[\T_l]:=\argmin\flp{\CC(\alpha,\B):\,\, (\alpha,\B)\in\T_l}
\ee
precisely. The following theorem is established in order to achieve \eqref{accptable_train_result_intro}.
\begin{theorem}[see Theorem \ref{PV_finiteapprox_result}]\label{PV_approx_result_intro}
Let a training ground $\mathbb T$ satisfies above box-constraint. Then the following assertions hold:
\begin{enumerate}[1.]
\item
we have
\be
\lim_{l\to\infty}\operatorname{dist}(\mathbb A[\T],\mathbb A[\T_l])=0.
\ee
\item
Let $\e>0$ be given. Then for each $l\in\N$ we have 
\be
\abs{\CC(\alpha_{\T_l},\B_{\T_l})-\CC(\alpha_\T,\B_\T)}\leq 4KP \fmp{O\fsp{P/l}+1/l}^{1/2}\norm{u_\eta}_{W^{d,1}(Q)}^{1/2}/\e^d+\e/2,
\ee
where the value of right hand side can be computed explicitly.
\end{enumerate}
\end{theorem}
That is, for any given $\e>0$, we could compute $l\in\N$ that is large enough so that the corresponding optimal solution $(\alpha_{\T_l},\B_{\T_l})$ is 
an acceptable optimal solution of scheme $\FT$. Also, in Section \ref{primal_dual_rmk} we show that, even with the box-constraint, the training ground $\T$ is still sufficiently large to
encompass interesting operator. We finally remark that, although this work focuses mainly on the theoretical analysis of the operators $PV_\B$ and  the training scheme $\mathcal T$, in Section \ref{primal_dual_rmk} a primal-dual algorithm for solving \eqref{result_level_bi} is discussed, and some preliminary numerical demonstration of scheme $\mathcal T$ are provided.\\\\
Our article is organized as follows. In Section \ref{PABBQ_sec} we analyze the functional properties of the $PV_\B$-seminorms.
The $\Gamma$-convergence result, the bilevel training scheme, and the finite approximation  are the subjects of Sections \ref{sec_PDE_const} and \ref{sec_ts_PGV}, respectively. 
Finally, in Section \ref{primal_dual_rmk} we demonstrate several numerical implementations, and in Section \ref{sec_further_generaliza} some possible extensions of $PV_\B$.
\section{The space of functions with bounded $PV$-seminorm}\label{PABBQ_sec}
Let $d$, $N\in\N$ be given, and let $Q:=(0,1)^N$ be the unit open cube in $\rn$. $\mathbb M^{N_n}$ is the space of matrices with dimension $N\times N\times\cdots \times N$ ($n$ times) with elements in $\R$. For the convenience of the presentation of this article, we identify the matrix space $\mathbb M^{N_n}$ by vector space $\R^{N^n}$, where $N^n=N\cdot N\cdots N$ ($n$ times). Moreover, $\mathcal D'(Q, \R^n)$ represents the space of distributions with values in $\R^n$.
 \begin{notation}\label{operator_B}
We  collect some notation which will be adopted in connection with linear differential operators.
 \begin{enumerate}[1.] 
 \item
For $h\in\N$, we let $H^h$: $\mathcal D'(Q)\to\mathcal D'(Q;\mathbb R^{N^h})$, denote the $h$-th \emph{Hessian} differential operator. For example, when $h=1$ we have $H^1 u =\nabla u$;
\item
For each $h=1,\ldots d$, we let $B^h$ be matrix mapping from $\R^{N^h}$ to $\R^{N^h}$ and 
\be
K:=\sum_{h\in\N,h\leq d}N^h.
\ee
We denote by $\B$: $\mathcal D'(Q)\to \mathcal D'(Q;\rk)$ the $d$-th order differential operator
\be\label{A_quasiconvexity_operator}
\B u:= \sum_{h\in\N,h\leq d}B^h (H^hu);
\ee
\item
For each $h=1,\ldots d$, we denote the formal adjoint of the matrix $B^h$ by $(B^h)^\ast$, and we define the differential operator $\B^\ast$: $\mathcal D'(Q; \rk)\to\mathcal D'(Q)$ by
\be
\fjp{\B^\ast v,u}_{\R}:=\fjp{v,\B u}_{\R^{K}};
\ee
\item
We denote the bilinear operator $\circ_\B$, induced by $\B$, such that
\be\label{distribution_product}
\B(uw)=w\B u+u\circ_\B  w
\ee
\item
Given a sequence of operators $\seqn{\B_n}$ and an operator $\B$, with coefficients $\seqn{B_n}$ and $B$, respectively, we say that $\B_n\to\B$ in $\ell^{\infty}$ if
\be
\abs{\B_n-\B}:=\sum_{h\leq d}\abs{B^h_n-B^h}_{\ell^\infty}\to 0,
\ee
where $\abs{\cdot}_{\ell^\infty}$ stands for the $\ell^\infty$ matrix norm.
\end{enumerate}
\end{notation}
\begin{define}\label{def_operator_B}
Let $d\in\N$ be fixed. We denote by $\Pi^d$ the collection of operator $\B$ defined in notation \ref{operator_B}, with order at most $d$.
\end{define}
\subsection{The PDE-constraint total variation defined by operator $\B$}
We generalize the standard total variation seminorm by using the $d$-th order differential operators $\B\in\Pi^d$ defined in Definition \ref{def_operator_B}.
\begin{define}\label{def_BVB}
Let $u\in L^1(Q)$ and operator $\B\in\Pi^d$ be given. 
\begin{enumerate}[1.]
\item
We define the \emph{PDE-constraint} seminorm, say $PV_\B$, by
\be\label{BVB_norm}
PV_\B(u):=\sup\flp{\int_Q u\,\B^\ast\vp\,dx:\,\,\vp\in C_c^\infty(Q;\rk),\,\,\abs{\vp}\leq 1};
\ee
\item
We define the space
\be
BV_{\B}(Q):=\flp{u\in L^1(Q):\,\, PV_{\B}(u)<+\infty},
\ee
and we equip it with the norm
\be\label{BVB_norm}
\norm{u}_{BV_\B(Q)}:=\norm{u}_{L^1(Q)}+PV_\B(u).
\ee
\end{enumerate}
\end{define}
In next proposition we collect several preliminary results regarding functions in space $BV_\B(Q)$.
\begin{proposition}\label{lower_semicontity}
Let operator $\B\in\Pi^d$ and $u\in BV_\B(Q)$ be given. 
\begin{enumerate}[1.]
\item
For any sequence $\seqn{u_n}\subset L^1(Q)$ and function $u\in L^1(Q)$ that satisfying one of the following conditions:
\begin{enumerate}[i.]
\item
$\seqn{u_n}$ is locally uniformly integrable and $u_n\to u$ a.e.. 
\item
$u_n\wtos u$ in $\mb(Q)$.
\end{enumerate}
Then, we have
\be\label{eq_liminf_b}
\liminfn\, PV_\B(u_n)\geq PV_\B(u).
\ee
\item
There exists a \emph{Radon} measure $\mu$ and a $\mu$-measurable function $\sigma$: $Q\to \rk$ such that 
\begin{enumerate}[i.]
\item
$\abs{\sigma(x)}=1$ $\mu$-a.e.;
\item
for all $\vp\in C_c^\infty(Q;\rk)$, there holds
\be
\int_Q u\,\B^\ast \vp\,dx=-\int_Q\vp\cdot\sigma\,d\mu.
\ee
\end{enumerate}
\end{enumerate}
\end{proposition}
\begin{proof}
We prove Assertion 1 first. If 
\be
\liminfn\, PV_\B(u_n)=+\infty,
\ee
there is nothing to prove. Assume not, then we have, for arbitrary $\vp\in C_c^\infty(Q;\rk)$, that
\be
\liminfn\, PV_\B(u_n)\geq \liminfn\int_Q u_n\B^\ast\vp\,dx = \int_Q u\,\B^\ast\vp\,dx,
\ee
where the last equality can be deduced either from condition 1(i) or 1(ii), independently. Hence, we conclude \eqref{eq_liminf_b} in view of the arbitrariness of $\vp$. \\\\
We next prove Assertion 2. We define the linear functional $L$: $C_c^\infty(Q;\rk)\to\R$ such that 
\be
L(\vp):=-\int_Q u\,\B^\ast\vp\,dx,\text{ for }\vp\in C_c^\infty(Q;\rk).
\ee
Then, since $u\in BV_\B(Q)$, we have that
\be
\sup\flp{\frac1{\norm{\vp}_{L^\infty(Q)}}\int_Qu\, \B^\ast \vp \,dx:\text{ for }\vp\in C_c^\infty(Q;\rk)}=PV_\B(u)<+\infty,
\ee
which implies that
\be\label{smooth_define}
\abs{L(\vp)}\leq PV_\B(u)\norm{\vp}_{L^\infty(Q)}.
\ee
Now, for arbitrary $\vp\in C_c(Q;\rk)$, we define the mollifications $\vp_\e:=\vp\ast\eta_\e$, for some mollifier $\eta_\e$ with $\e<\operatorname{dist}(\operatorname{spt}(\vp),\partial Q)$. Then we have, by \cite[Theorem 1, item (ii), Section 4.2]{evans2015measure}, that $\vp_\e\to\vp$ uniformly on $Q$. Therefore, by defining
\be
\bar L(\vp):=\lime L(\vp_\e)\text{ for }\vp\in C_c(Q;\rk),
\ee
and together with \eqref{smooth_define}, we conclude that 
\be
\sup\flp{\bar L(\vp):\text{ for }\vp\in C_c(Q;\rk)\text{ and }\abs{\vp}\leq 1}<+\infty.
\ee
Thus, in view of the Riesz representation theorem (see \cite[Section 1.8]{evans2015measure}), the proof is complete.
\end{proof}
\begin{remark}\label{thm_structure_represent}
We henceforth write $\abs{\B u}$ by $\mu$ and have
\be
\int_Q u\,\B^\ast \vp\,dx = -\int_Q \vp\cdot\sigma\,d\abs{\B u}
\ee
for arbitrary $\vp\in C_c^\infty(Q;\rk)$.
\end{remark}
\begin{theorem}[local approximation by smooth functions]\label{smooth_strict_approximation}
Let $p\geq 1$ and $u\in BV_\B(Q)\cap L^p(Q)$ be given. There exists a sequence $\seqn{u_n}\subset C^\infty(Q)\cap BV_\B(Q)$ such that the following assertions hold.
\begin{enumerate}[1.]
\item
$u_n\to u$ strongly in $L^p(Q)$;
\item
$PV_{\B}(u_n)\to PV_\B(u)$.
\item
$u_n\in C^\infty(\bar Q)$ for each $n\in\N$.
\end{enumerate}
\end{theorem}
\textbf{Remark.} Assertion 3 only asserts that for each fixed $n\in \N$ that $u_n\in C^\infty(\bar Q)$ but it is possible that $\norm{u_n}_{L^1(\partial Q)}\to\infty$ as $n\to\infty$. In another word, we make no conclusions with respect to the trace value of $u$ from Theorem \ref{smooth_strict_approximation}. 
\begin{proof}
The construction of approximation sequence $\seqn{u_n}$ is almost same to the approximation sequence used in the standard $BV$ case as presented in \cite[Theorem 2, Page 172]{evans2015measure}. We shall only concentrated on showing that Assertion 3 holds, but for reader's convenience, we shall outline the  construction of approximation sequence and key steps.\\\\
Let $u\in BV_\B(Q)$ be given, and let $Q_k$ be the cube centered at point $q=(1/2,1/2)^N$ with side length $1-1/(k+M)$. Let arbitrary $\e>0$ be given, we choose $M>0$ large enough such that  
\be
\abs{\B u}(Q\setminus Q_{0+M})<\e/2.
\ee
Define $Q_0=Q_{0+M}$ and
\be
V_k:=Q_{k+1}\setminus \bar Q_{k-1}\text{ for }k\in\N.
\ee
Let $\seqk{\zeta_k}\subset C_c^\infty(Q)$ be the partition of unity such that 
\begin{align}
&\zeta_k\in C_c^\infty(V_k)\text{ such that }0\leq \zeta_k\leq 1\\
&\sum_{k\ge1}\zeta_k(x)=1\text{ for each }x\in Q.
\end{align}
Let $\eta_\e$ be the standard mollifier, and for each $k$, we choose $\e_k$ small enough such that 
\begin{align}
& \operatorname{spt}(\eta_{\e_k}\ast(u\,\zeta_{k}))\subset V_k\label{smooth_function_prop}\\
& \norm{\eta_{\e_k}\ast(u\,\zeta_k)-u\,\zeta_k}_{L^p(Q)}<\e/2^{k+1}\label{l1_approach_prop}\\
& \norm{\eta_{\e_k}\ast(u\B\zeta_k)-u\B\zeta_k}_{L^1(Q)}<\e/2^{k+1}\label{TV_approach_prop},
\end{align}
and we define
\be
u_\e:=\sum_{k=1}^\infty \eta_{\e_k}\ast(u\zeta_k).
\ee
We observe that \eqref{smooth_function_prop} implies that $u_\e\in C^\infty(Q)$, and \eqref{l1_approach_prop} implies that 
\be
u_\e\to u\text{ strongly in }L^p(Q).
\ee
This, and together with Assertion 1, Proposition \ref{lower_semicontity}, we conclude that
\be
\liminf_{\e\to 0}PV_\B(u_\e)\geq PV_\B(u).
\ee
Next, for arbitrary $\vp\in C_c^\infty(Q;\rn)$, we observe that, 
\be
\fjp{\eta_{\e_k}\ast(u\,\zeta_k),\B^\ast \vp} = \fjp{u\,\zeta_k,\B^\ast (\eta_{\e_k}\ast\vp)} = \fjp{u,\B^\ast(\zeta_k(\eta_{\xi_k}\ast\vp))}-\fjp{u,(\eta_{\e_k}\ast\vp)\circ_{\B^\ast}\nabla \zeta_k},
\ee
where at the first equality we used the linearity of convolution operator, and at the last equality we used \eqref{distribution_product}. Thus, we have
\be
\fjp{u_\e,\B^\ast\vp} = \sum_{k\geq 1}\fjp{\eta_{\e_k}\ast(u\zeta_k),\B^\ast \vp} =  \sum_{k\geq 1}\fjp{u,\B^\ast(\zeta_k(\eta_{\xi_k}\ast\vp))}- \sum_{k\geq 1}\fjp{u,(\eta_{\e_k}\ast\vp)\circ_{\B}\nabla \zeta_k}.
\ee
Following the same computation used in \cite[Theorem 2, Page 172]{evans2015measure} and use \eqref{TV_approach_prop}, we deduce that 
\be
\fjp{u_\e,\B^\ast\vp}\leq PV_\B(u)+\e,
\ee
Hence, in view of the arbitrariness of $\vp$, we obtain that
\be
\limsup_{\e\to 0} PV_\B(u_\e)\leq PV_\B(u),
\ee
Lastly, we further modify the sequence $\seqe{u_\e}$ so that $u_\e\in C^\infty(\bar Q)$ for each $\e>0$. Let $\delta>0$ be given and we define 
\be\label{easy_checting_u}
u_{\e,\delta}(x):=u_\e((x-q)/(1+\delta))\text{, for }x\in Q.
\ee
Consequentially, we have $u_{\e,\delta}\to u_\e$ in $L^p$ strong and $PV_\B(u_{\e,\delta})\to PV_\B(u_\e)$, as $\delta\to 0$. Hence, by using a diagonal argument, we could extract a subsequence $\seqe{u_{\delta_\e}}$ such that 
\be
u_{\delta_\e}\to u\text{ strongly in }L^p\text{ and }PV_\B(u_{\delta_\e})\to PV_\B(u). 
\ee
On the other hand, by the definition of $u_{\delta_\e}$, we have $u_{\delta_\e}\in C^\infty(\bar Q)$, which concludes Assertion 3 as desired.
\end{proof}
\textbf{Remark.} The construction of $u_{\e,\delta}$ in \eqref{easy_checting_u} is possible because of the simple geometry of domain $Q$. However, for domain with arbitrary geometry, even with Lipschitz boundary, such construction is not available. We refer readers to \cite{bbddlgftfb2017,fonseca1993relaxation} for alternative constructions with, however, operator $\B$ with several additional restriction. 
\begin{corollary}\label{smooth_strict_approximation_finite}
Let a finite set of $\B_i$, $i=1,\ldots M$, be given and 
\be
u\in \bigcap_{i=1}^M BV_{\B_i}(Q).
\ee
Then, there exists a sequence 
\be
\seqn{u_n}\subset C^\infty(Q)\cap \bigcap_{i=1}^M BV_{\B_i}(Q)
\ee
 such that the following assertions hold.
\begin{enumerate}[1.]
\item
$u_n\to u$ strongly in $L^1(Q)$;
\item
$PV_{\B_i}(u_n)\to PV_{\B_i}(u)$, for each $i=1,\ldots,M$ uniformly;
\item
$u_n\in C^\infty(\bar Q)$ for each $n\in\N$.
\end{enumerate}
\end{corollary}
\begin{proof}
We only need to change \eqref{TV_approach_prop} to 
\be
\sum_{i=1}^M\norm{\eta_{\e_k}\ast(u\B_i\zeta_k)-u\B_i\zeta_k}_{L^1(Q)}<\e/2^{k+1},
\ee
and the rest follows with the same argument used in the proof of Theorem \ref{smooth_strict_approximation}.
\end{proof}

We close this section by stating the \emph{l.s.c.} result of $PV_\B$ semi-norm.
\begin{proposition}\label{thm_asymptitic_PGV}
Let $u\in L^1(Q)$ and sequence $\seqn{\B_n}$ such that $\B_n\to \B$ in $\ell^{\infty}$ be given. Then, we have that
\be
\limn PV_{\B_n}(u)\geq PV_{\B}(u).
\ee
\end{proposition}
\begin{proof}
First of all, if 
\be
\liminfn \,PV_{\B_n}(u)=+\infty,
\ee
then there is nothing to prove. Suppose 
\be
\sup\flp{PV_{\B_n}(u):\,\,n\in\N}:=M<+\infty,
\ee
then, for arbitrary $\vp\in C_c^\infty(Q;\rk)$, we have that
\be
+\infty>\liminfn PV_{\B_n}(u)\geq\liminfn\int_Q u\,\B_n^\ast\vp\,dx = \int_Qu\,\B^\ast\vp\,dx.
\ee
Hence, by taking supremum with respect to $\vp$ on the right hand side of above inequality, we conclude that 
\be
\liminfn PV_{\B_n}(u)\geq PV_{\B}(u),
\ee 
as desired.
\end{proof}


\section{Analytic properties of PDE-constraint variations}\label{sec_PDE_const}
\subsection{$\Gamma$-convergence of functionals defined by $PV$ seminorms}\label{gamma_conv_sec}
In this section we prove a $\Gamma$-convergence result with respect to the intensity parameter $\alpha$ and operator $\B$. 
\begin{define}
\label{def:fractional TGV functional}
We define the functional $\mathcal I_{\alpha,\B}$ :$L^1(Q)\to [0,+\infty]$ as
\be
\mathcal I_{\alpha,\B}(u):=
\begin{cases}
\norm{u-u_\eta}_{L^2(Q)}^2+ \alpha PV_\B(u)&\text{ if }u\in BV_\B(Q),\\
+\infty &\text{ otherwise. }
\end{cases}
\ee
\end{define}
The following theorem is the main result of this section.
\begin{theorem}\label{thm_Gamma_conv}
Let sequences $\seqn{\B_n}$ and $\seqn{\alpha_n}$ be given such that $\B_n\to \B_0$ in $\ell^{\infty}$ and $\alpha_n\to \alpha_0\in\R^+$. Then, the functional $\mathcal I_{\alpha_n,\B_n}$ $\Gamma$-converges to $\mathcal I_{\alpha,\B}$ in the weak $L^2$ topology. To be precise, for every $u\in L^1(Q)$ the following two conditions hold:\\

\emph{(Lower semi-continuity)} If
\be u_n\wto u\text{ weakly in }L^2(Q)\ee
then 
\be
\mathcal I_{\alpha,\B}(u)\leq \liminf_{n\to +\infty}\mathcal I_{\alpha_n,\B_n}(u_n).
\ee
\emph{(Recovery sequence)} For each $u\in BV(Q)$, there exists $\seqn{u_n}\subset L^1(Q)$ such that 
\be u_n\wto u\text{ weakly in }L^2(Q)\ee
and 
\be
\limsup_{n\to +\infty}\,\mathcal I_{\alpha_n,\B_n}(u_n)\leq \mathcal I_{\alpha,\B}(u).
\ee
\end{theorem}
We subdivide the proof of Theorem \ref{thm_Gamma_conv} into two propositions.\\\\
The following proposition is instrumental for establishing the liminf inequality.
\begin{proposition}\label{compact_semi_para}
Let sequences $\seqn{\B_n}$ and $\seqn{\alpha_n}$ be given such that $\B_n\to \B_0$ in $\ell^{\infty}$ and $\alpha_n\to \alpha_0\in\R^+$. Let $\seqn{u_n}\subset L^1(Q)$ be given such that there exists $p\in(1,+\infty]$ and 
\be\label{liminf_unf_upper}
\sup\flp{\norm{u_n}_{L^p(Q)}+PV_{\B_n}(u_n):\,\, n\in\N}<+\infty.
\ee
Then there exists $u_0\in BV_{\B_0}(Q)$ such that, up to the extraction of a  subsequence (not relabeled),
\be\label{weak_BV_liminf}
u_n\wto  u_0\text{ weakly}\text{ in }L^p(Q)
\ee
and 
\be\label{lsc_PV_Bn}
\liminf_{n\to\infty}PV_{\B_n}(u_n)\geq PV_{\B_0}(u_0).
\ee
\end{proposition}
\begin{proof}
Without loss of generality  we assume that $\alpha_n=1$ for every $n\in\N$, as the general case for $\alpha_n$ and $\alpha_0\in \R^+$ can be argued with straightforward adaptations. \\\\
From \eqref{liminf_unf_upper} and the fact $p>1$ we have, up to a subsequence, that there exists $u_0\in L^p(Q)$ such that \eqref{weak_BV_liminf} holds.\\\\
Next, for arbitrary $\vp\in C_c^\infty(Q;\rk)$, we observe that
\begin{align}\label{cuowei_liminf}
&\limsup_{n\to\infty}\abs{\int_Q u_n\B^\ast_n\vp\,dx-\int_Q u_n\B^\ast_0\vp\,dx}\\
&\leq \limsup_{n\to\infty} \int_Q \abs{u_n}\abs{\B^\ast_n\vp-\B^\ast_0\vp}\,dx\leq \fsp{\sup_{n\geq 0}\norm{u_n}_{L^p}}\fsp{\limsup_{n\to\infty}\norm{\B^\ast_n\vp-\B^\ast_0\vp}_{L^{p'}}}= 0,
\end{align}
where at the last we used the fact that $\vp\in C_C^\infty(Q;\rk)$ and the Lebesgue dominated convergence theorem.\\\\
Hence, we could obtain that 
\begin{align}
&\liminfn PV_{\B_n}(u_n)\geq \liminfn\int_Qu_n\B_n^\ast\vp\,dx\\
 &\geq  \liminfn\int_Qu_n\B_0^\ast\vp\,dx+ \liminfn\int_Qu_n(\B_n^\ast-\B_0^\ast)\vp\,dx\geq  \int_Qu_0\B_0^\ast\vp\,dx,
\end{align}
where at the last inequality we used \eqref{weak_BV_liminf} and \eqref{cuowei_liminf}. Thus, by the arbitrarness of $\vp\in C_c^\infty(Q;\rk)$, we conclude \eqref{lsc_PV_Bn}, and hence the thesis.
\end{proof}
\begin{proposition}[$\Gamma$-$\limsup$ inequality]\label{new_equal}
Let sequences $\seqn{\B_n}$ and $\seqn{\alpha_n}$ be given such that $\B_n\to \B_0$ in $\ell^{\infty}$ and $\alpha_n\to \alpha_0\in\R^+$. Then, for every $u_0\in BV_{\B_0}(Q)$ there exist $\seqn{u_n}\subset BV_{\B_n}(Q)$ and, up to a subsequence of $\seqn{\B_n}$, such that $u_n\wto u_0$ in $L^p$ and
\be
\limsup_{n\to\infty}PV_{\B_n}(u_n)\leq PV_{\B_0}(u_0).
\ee
\end{proposition}
\begin{proof}
If $PV_{\B_0}(u)=\infty$, there is nothing to prove. Suppose not, and assume for a moment that $u_0\in C^\infty(\bar Q)$, which indicates that $u_0\in BV_{\B_n}(Q)$ for each $n\in\N$. Fix $\delta>0$, and chose $\vp_{\delta,n}\in C_c^\infty(Q;\rk)$ such that 
\be\label{limsup_Bn1}
PV_{\B_n}(u)\leq \int_Qu\,\B_n^\ast \vp_{\delta,n} dx+\delta.
\ee
We observe that
\be\label{limsup_Bn2}
\abs{\int_Qu_0\B_n^\ast \vp_{\delta,n} dx} = \abs{\int_Q[\B_nu_0]\vp_{\delta,n}dx}\leq \norm{\vp_{\delta,n}}_{L^\infty(Q)}\int_Q\abs{\B_nu_0}dx\leq \int_Q\abs{\B_n u_0}dx,
 \ee
where at the last inequality we used the fact that $\vp_{\delta,n}$ satisfies \eqref{BVB_norm}. Next, since $u\in C^\infty(\bar Q)$ and $\B_n\to\B$ in $\ell^\infty$, we have
\be
\abs{\B_nu(x)}\leq \sup\flp{\abs{\B_n}_{\ell^\infty}:\,\,n\in\N}\cdot\sum_{h\leq d}\abs{H^h u_0(x)}_{\ell^\infty},
\ee
which implies that
\be
\int_Q\sum_{h\leq d}\abs{H^l u(x)}_{\ell^\infty}dx\leq \norm{u}_{W^{d,+\infty}(Q)}<+\infty.
\ee
Thus, we could apply the Lebesgure dominate convergence theorem to conclude that 
\be
\limsup_{n\to\infty}\int_Q\abs{\B_nu_0}dx\leq \int_Q\limsup_{n\to\infty}\abs{\B_nu_0}dx = \int_Q\abs{\B_0u_0}dx.
\ee
This, and together with \eqref{limsup_Bn1} and \eqref{limsup_Bn2}, we observe that 
\be
\limsup_{n\to\infty} PV_{\B_n}(u_0)\leq \limsup_{n\to\infty}\int_Qu_0\B_n^\ast \vp_{\delta,n} dx+\delta\leq \int_Q\abs{\B_0u_0}dx+\delta=PV_{\B_0}(u_0)+\delta,
\ee
which implies, by sending $\delta\searrow 0$ second, that
\be\label{smooth_limsup}
\limsup_{n\to\infty} PV_{\B_n}(u_0)\leq PV_{\B_0}(u_0).
\ee
Next, by Theorem \ref{smooth_strict_approximation}, we could construct an approximation sequence $\seqe{u_\e}\subset C^\infty(\bar Q)$ such that $u_\e\to u$ in $L^p(Q)$ and
\be
PV_{\B_0}(u_\e)\to PV_{\B_0}(u),\text{ or }PV_{\B_0}(u_\e)\leq PV_{\B_0}(u)+O(\e).
\ee
Also, by \eqref{smooth_limsup}, we have
\be
\limsup_{n\to\infty}PV_{\B_n}(u_\e)\leq PV_{\B_0}(u_\e)\leq PV_{\B_0}(u)+O(\e).
\ee 
Thus, by a diagonal argument, we can obtain a sequence $\seqe{\B_{n_\e}}$ such that 
\be
PV_{\B_{n_\e}}(u_\e)\leq PV_{\B_0}(u)+O(\e).
\ee
That is, we have 
\be
\limsup_{\e\to 0}PV_{\B_{n_\e}}(u_\e)\leq PV_{\B_0}(u),
\ee
which concludes our thesis.
\end{proof}

We close Section \ref{gamma_conv_sec} by proving Theorem \ref{thm_Gamma_conv}.
\begin{proof}[Proof of Theorem \ref{thm_Gamma_conv}]
Property \emph{(Lower semi-continuity)} hold in view of Proposition \ref{compact_semi_para}, and Property \emph{(Recovery sequence)} follows from Proposition \ref{new_equal}. 
\end{proof}

\subsection{The point-wise characterization of sub-differental of $PV_\B$}\label{pointwise_chara_sd}
We recall few notations and preliminary results and definitions first. 
\begin{define}[{\cite[Definition 4.1 \& 5.1]{ekeland1976convex}}]\label{conjugatefunction}
Let $F$ be a function of normed space $V$ into $\bar \R$ be given.
\begin{enumerate}[1.]
\item
We define the \emph{polar} function of $F$, denoted by $F^\ast$, by
\be
F^\ast(u^\ast)=\sup\flp{\fjp{v,u^\ast}_{V,V^\ast}-F(v):\,\,v\in V}.
\ee
\item
We define the \emph{bipolar} function, say $F^{\ast\ast}$, of $F$ by 
\be
F^{\ast\ast}=(F^\ast)^\ast.
\ee
\item
We say $F$ is \emph{sub-differentiable} at point $u\in V$ if $F(u)$ is finite and there exists $u^\ast\in V^\ast$ such that 
\be
\fjp{v-u,u^\ast}_{V,V^\ast}+F(u)\leq F(v)
\ee
for all $v\in V$. Then we call such $u^\ast\in V^\ast$ is called a \emph{sub-gradient} of $F$ at $u$, and the set of sub-gradients at $u$ is called the \emph{sub-differential} at $u$ and is denoted $\partial F(u)$.
\end{enumerate}
\end{define}

\begin{proposition}[{\cite[Proposition 4.1 \& 5.1]{ekeland1976convex}}]\label{prop_ekeland1976convex}
Let $F$ be a function of $V$ into $\bar \R$ and $F^\ast$ its polar. Then the following assertions hold.
\begin{enumerate}[1.]
\item
We have $u^\ast\in\partial F(u)$ if and only if
\be
F(u)+F^\ast(u^\ast)=\fjp{u,u^\ast}.
\ee
\item
The set $\partial F(u)$ (possible empty) is convex and closed.
\item
If in addition $F$ is convex, then $F^{\ast\ast}=F$. 
\end{enumerate}
\end{proposition}

\begin{define}
Let $p\in[1,+\infty)$, $v\in L^p(Q;\rk)$, and operator $\B$ be given. 
\begin{enumerate}[1.]
\item
we say that $\B^\ast v$ in $L^p(Q)$ if there exists $w\in L^p(Q)$ such that for all $\vp\in C_c^\infty(Q;\rk)$
\be
\int_Q\B\vp\cdot v\,dx=-\int_Q \vp w\,dx.
\ee
\item
we define the space
\be
W^p[\B](Q;\rk):=\flp{v\in L^p(Q;\rk):\,\,\B^\ast v\in L^p(Q)}
\ee
with the norm
\be
\norm{v}_{W^p(\B)}^p:=\norm{v}_{L^p(Q)}^p+\norm{\B^\ast v}_{L^p(Q)}^p.
\ee
\item
we define
\be
W_0^p[\B](Q;\rk):=\operatorname{cl}(C_c^\infty(Q;\rk))_{\norm{\cdot}_{W^p(\B)}},
\ee
i.e., the closure of function space $C_c^\infty(Q;\rn)$ with respect to $\norm{\cdot}_{W^p(\B)}$ norm.
\item
we define
\be
C_c^\infty[\B](Q):=\flp{\B^\ast \vp:\,\,\vp\in C_c^\infty(Q;\rk),\,\,\norm{\vp}_{L^\infty(Q)}\leq 1}.
\ee
and
\be
K^p[\B](Q):=\flp{\B^\ast v:\,\,v\in W^p_0[\B](Q;\rk),\,\,\norm{v}_{L^\infty(Q)}\leq 1}.
\ee

\end{enumerate}
\end{define}
The main result of Section \ref{pointwise_chara_sd} reads as follows.
\begin{theorem}\label{main_sub_diff_pvb}
Let $p>1$, $q=p/(p-1)$, and $u\in L^p(Q)$, $\tilde u\in L^{q}(Q)$. Then $\tilde u\in \partial PV_\B(u)$ if and only if
\begin{enumerate}[1.]
\item
$u\in BV_\B(Q)$;
\item
there exist $v\in W_0^{q}[\B](Q;\rk)$ such that $\norm{v}_{L^\infty(Q)}\leq 1$, $\tilde u=\B^\ast v$,  and 
\be
PV_\B(u)=\int_Q u\,\B^\ast v\,dx.
\ee
\end{enumerate}
\end{theorem}
We prove Theorem \ref{main_sub_diff_pvb} in several propositions.
\begin{proposition}\label{approxi_sub_space}
Let $p\in(1,+\infty)$ be given. Then we have the closure of function space $C_\B(Q)$ under $L^q$ norm equals to the function space $W^q_0[\B](Q)$, i.e.,
\be
\operatorname{cl}(C_\B(Q))_{L^q}=W^q_0[\B](Q).
\ee
\end{proposition}
\begin{proof}
We claim
\be\label{subset_direction1}
\operatorname{cl}(C_\B(Q))_{L^q}\subset W^q_0[\B](Q)
\ee
first, and we do it by showing the space $W^q_0[\B](Q)$ is closed with respect to $L^q$ norm. Let $g\in \operatorname{cl}(W^q[\B](Q;\rk))$ be given, and extract a sequence $\seqn{v_n}\subset W^q_0[\B](Q;\rk)$ such that 
\be\label{closeness_asumption}
\norm{\B^\ast v_n-g}_{L^q(Q)}\to 0.
\ee
Since  $\seqn{v_n}\subset W^q[\B](Q;\rk)$, we have $\norm{v_n}_{L^\infty}\leq 1$ and hence, up to a subsequence, there exists $v_0\in L^\infty$ such that 
\be
v_n\wto v_0\text{ weakly in }L^q\text{ and }\norm{v_0}_{L^\infty}\leq 1.
\ee 
Next, let $\phi\in C_c^\infty (Q)$ be given, and we observe that 
\be
\int_Q \B^\ast v_n\phi\,dx=-\int_Qv_n\B\phi\,dx\to -\int_Q v_0\B\phi\,dx,
\ee
and together with \eqref{closeness_asumption}, we have
\be
\int_Q g\phi\,dx=-\int_Qv_0\B\phi\,dx,
\ee
which implies that $g=\B^\ast v_0$. Thus, we have $v_0\in W_0^q[\B](Q;\rk)$. Next, since the set
\be
\flp{(v,\B^\ast v):\,\,v\in W^q_0[\B](Q;\rk)}\subset L^q(Q;\rk\times \R)
\ee
is convex and closed, hence by \cite[Theorem 3.7]{brezis2010functional}, it is weakly closed. Thus, we conclude that $v_0\in W^q_0[\B](Q)$, which implies that $g\in W^q_0[\B](Q)$, or the function space $W^q_0[\B](Q)$ is closed with respect to $L^q$ norm, which also conclude \eqref{subset_direction1} as desired.\\\\
We next claim that 
\be\label{subset_direction2}
\operatorname{cl}(C_\B(Q))_{L^q}\supset W^q_0[\B](Q).
\ee
We prove \eqref{subset_direction2} by following arguments used in \cite[Theorem 2, Page 125]{evans2015measure}. Let $g\in W^q_0[\B](Q)$ be given. That is, there exists $v\in W^q_0[\B](Q;\rk)$, $\norm{v}_{L^\infty}\leq 1$, and $g=\B^\ast v$. From the definition of $W^q[\B](Q)$, there exists $\seqn{v_n}\subset C_c^\infty(Q;\rk)$ such that
\be\label{first_conv_eq}
\norm{v_n-v}_{W^q[\B](Q)}\to 0.
\ee
Next, define the truncation function
\be\label{chopper_off}
\bar v_n:=-1\vee v_n\wedge 1,
\ee
and we note that
\be\label{ae_weakly_to}
\bar v_n\to v\text{ a.e., and }\B^\ast \bar v_n\wto \B^\ast v_0\text{ weakly in }L^q.
\ee
Using a similar argument used in Proposition \ref{lower_semicontity}, and together with $\bar v_n\to v$ a.e., we obatin that 
\be
\liminfn \norm{\B^\ast \bar v_n}_{L^q(Q)}\geq \norm{\B^\ast v_0}_{L^q(Q)}.
\ee
On the order hand, by \eqref{chopper_off}, we have 
\be
\norm{\B^\ast \bar v_n}_{L^q(Q)}\leq \norm{\B^\ast v_0}_{L^q(Q)},
\ee
and hence 
\be
\limn\norm{\B^\ast \bar v_n}_{L^q(Q)}= \norm{\B^\ast v_0}_{L^q(Q)}.
\ee
This, and together with the second part in \eqref{ae_weakly_to}, and using \cite[Exercise 4.19, 1, page 124]{brezis2010functional}, we conclude that 
\be
\limn \norm{\B^\ast \bar v_n-\B^\ast v_0}_{L^q(Q)}=0.
\ee
We next modify the sequence $\seqn{\bar v_n}$ so that $\seqe{\B^\ast\bar v_n}\subset C_c^\infty[\B](Q)$.  We obtain sequence of sets $V_k$, $k\in\N$, and partition of unity $\zeta_k\in C_c^\infty(Q)$ from the argument used in Theorem \ref{smooth_strict_approximation}. Next, for each $k$, we choose $\e_k$ small enough such that 
\begin{align}
& \operatorname{spt}(\eta_{\e_k}\ast(\bar v_n\,\zeta_{k}))\subset V_k\label{smooth_function}\\
& \norm{\eta_{\e_k}\ast(\bar v_n\,\zeta_k)-\bar v_n\,\zeta_k}_{L^q(Q;\rn)}<\e/2^{k+1}\label{l1_approach}\\
& \norm{\eta_{\e_k}\ast(\B^\ast(\bar v_n\zeta_k))-\B^\ast(\bar v_n\zeta_k)}_{L^q(Q)}<\e/2^{k+1}\label{TV_approach},
\end{align}
and in addition, we choose $\e_k$ small that 
\be\label{small_spt_room}
\e_k\leq \operatorname{dist}(\partial Q,\operatorname{spt}(\bar v_n))/8.
\ee
Then, we define
\be
v_{\e,n}:=\sum_{k=1}^\infty \eta_{\e_k}\ast(\bar v_n\zeta_k)
\ee
and \eqref{small_spt_room} indicates that $v_{\e,n}\in C_c^\infty(Q)$. Following the same calculation in  \cite[Theorem 2, Page 125]{evans2015measure}, we have that 
\be
\lime\norm{v_{\e,n}-\bar v_n}_{W^q[\B](Q)}= 0,
\ee
together with \eqref{first_conv_eq}, and a diagonal argument, we could construct a sequence $\seqn{v_{\e_n}}$ such that 
\be
\lime\norm{v_{\e_n}- v}_{W^q[\B](Q)}= 0,
\ee
Moreover, we observe that 
\be
\abs{v_{\e_n}}\leq \abs{\sum_{k=1}^\infty \eta_{\e_k}\ast(\bar v_n\zeta_k)}\leq \abs{\bar v_n}\leq 1,
\ee
which proves that $\seqn{\B^\ast v_{\e_n}}\subset C_c^\infty[\B](Q)$, and hence \eqref{subset_direction2}, and our thesis.
\end{proof}
Now we ready to prove Theorem \ref{main_sub_diff_pvb}.
\begin{proof}[Proof of Theorem \ref{main_sub_diff_pvb}]
We first claim that the convex conjugate of $PV_\B$, say $PV_\B^\ast$, has the form that 
\be
PV_\B^\ast(v)=I_{W^q_0[\B](Q)}(v)=:
\begin{cases}
0&\text{ if }v\in W^q_0[\B](Q)\\
+\infty &\text{ if }v\notin W^q_0[\B](Q).
\end{cases}
\ee
By Definition \ref{conjugatefunction} and Proposition \ref{approxi_sub_space}, we have that 
\be
I^\ast_{W^q_0[\B](Q)}(u)=PV_\B(u).
\ee
Next, since the seminorm $PV_\B$ and indictor function $I_{\operatorname{cl}(C_\B(Q))_{L^q}}$ are convex and lower semi-continuity, we have 
\be
PV_\B^\ast(v)=(I_{W^q_0[\B](Q)}^{\ast})^\ast = I_{W^q_0[\B](Q)}.
\ee
Finally, in view of Proposition \ref{prop_ekeland1976convex}, we have that
\be
u^\ast\in\partial PV_\B(u)
\ee
if and only if
\be
PV_\B(u)+PV_\B^\ast(u^\ast) = \fjp{u,u^\ast},
\ee
and we are done.
\end{proof}
\begin{remark}\label{rmk_test_subgradient}
We have actually showed, in view of Proposition \ref{approxi_sub_space}, that for any $v\in W_0^p[\B](Q)$, $\norm{v}_{L^\infty}\leq 1$, the inequality 
\be
\int_Q u\,\B^\ast v\,dx\leq PV_\B(u)
\ee
holds.
\end{remark}

\begin{theorem}[The point-wise characterization of $\partial PV_\B$]\label{thm_point_chara}
Let $u\in L^p(Q)\cap BV_\B(Q)$, $p>1$, be given. Let $v\in W^p_0[\B](Q)$ such that ${\B^\ast} v\in \partial PV_\B(u)$. Then we have 
\be
v = \sigma_u \text{ a.e. }x\in Q,
\ee
where $\sigma_u$ is the density of $\B u$ with respect to $\abs{\B u}$ (see Remark \ref{thm_structure_represent}).
\end{theorem}
\begin{proof}
Let $u\in L^p(Q)\cap BV_\B(Q)$ be given and $v\in W^p_0[\B](Q)$ be obtained from Theorem \ref{main_sub_diff_pvb}. Then, by the definition of $W^p_0[\B](Q)$, we could obtain a sequence $\seqn{v_n}\subset C_c^\infty(\B^\ast,Q)$ such that ${\B^\ast}v_n\to {\B^\ast} v$ strongly in $L^p$.\\\\
We claim that 
\be\label{pointwise_chara}
\norm{\sigma_u-v_n}_{L^p(Q,\abs{\B u})}\to 0.
\ee
From the definition of $PV_\B$ and Theorem \ref{thm_structure_represent}, we have that 
\be\label{approxi_tvr}
\int_Q u\,{\B^\ast}v_n\,dx = \int_Qv_n \cdot\sigma_u d\abs{\B u}.
\ee
On the order hand, since $\seqn{v_n}\subset C_c^\infty(\B^\ast,Q)$, we have $\norm{v_n}_{L^\infty}\leq 1$ and hence, together with the fact that $\abs{\sigma_u}=1$ $\abs{\B u}$ a.e., we observe that 
\begin{align}
1-(\sigma_u\cdot v_n)&=\frac12\abs{\sigma_u}^2-(\sigma_u\cdot v_n)+\frac12\abs{v_n}^2+\frac12\abs{\sigma_u}^2-\frac12\abs{v_n}^2\\
&=\frac12\abs{\sigma_u-v_n}^2+\frac12\abs{\sigma_u}^2-\frac12\abs{v_n}^2\geq \frac12\abs{\sigma_u-v_n}^2\geq 0.
\end{align}
Therefore, we could compute that 
\begin{align}\label{approxi_upper}
&\int_Q \abs{v_n-\sigma_u}d\abs{\B u}=\int_Q 1\cdot\abs{v_n-\sigma_u}d\abs{\B u}\\
&\leq \fsp{\int_Q1d\abs{\B u}}^{1/2}\cdot \fsp{\int_Q \abs{v_n-\sigma_u}^2d\abs{\B u}}^{1/2}\\
&\leq \fmp{PV_\B(u)}^{1/2}\cdot \fsp{\int_Q 1-(\sigma_u\cdot v_n)d\abs{\B u} }^{1/2}.
\end{align}
Next, from \eqref{approxi_tvr}, we have that 
\be
\limn\int_Qv_n \cdot\sigma_u d\abs{\B u} = \int_Qu\,{\B^\ast}v\,dx = PV_\B(u) = \int_Q 1 d\abs{{\B} u}.
\ee
This, and together with \eqref{approxi_upper}, we conclude \eqref{pointwise_chara} as desired.
\end{proof}

\begin{proposition}\label{nested_sub_gradient}
Let $u\in BV_\B(Q)$ and $V\subset\subset Q$ be given. Let $u^\ast\in \partial PV_\B(u)$ and $ u_V^\ast\in \partial PV_\B(u)\lfloor_V$, then we have
\be
u^\ast(x) =  u_V^\ast(x)\text{ for }\abs{\B u}\text{-a.e. }x\in V.
\ee
\end{proposition}
\begin{proof}
We obtain $v$ and $v_V\in W_0^p[\B](Q;\rk)$ such that Assertions 1 and 2 hold for $PV_\B(u)$ and $PV_\B(u)\lfloor_V$, respectively. Then, by Theorem \ref{thm_point_chara} we have both $v(x)$ and $v_V(x)$ can be represented point-wisely by the density of $\B u$ with respect to $\abs{\B u}$, and we are done.
\end{proof}

\section{Learning the optimal operator $\B$ in imaging processing problems}\label{sec_ts_PGV}
In this section we use the bilevel training scheme introduced in Section \ref{sec_introduction} to determine the optimal setting of $PV_\B$ for a given \emph{training pairs} $(u_c,u_\eta)$, where $u_\eta\in L^2(Q)$ and $u_c\in BV(Q)$ represents the corrupted and clean image, respectively.

\subsection{The bilevel training scheme with the $PV_\B$ regularizer}

We collect few notations first.
\begin{notation}\label{def_training_set}
Recall the definition of $\B$ from Notation \ref{operator_B}. 
\begin{enumerate}[1.]
\item
We denote by $\Sigma$ the collection of operators $\B$ such that 
\be
\Sigma:=\flp{\B:\,\,\abs{ \B}_{\ell^\infty}\leq 1}
\ee
\item
We denote the \emph{TrainingGround}  $\mathbb T$ by
\be
\mathbb T:=\operatorname{cl}(\R^+)\times \Sigma.
\ee
 \end{enumerate}
\end{notation}
We state below the definition of training scheme $\mathcal T$ and associated notations. 
\begin{define}\label{training_T}
We define the training scheme $\mathcal T$ with underlying training ground $\mathbb T$ by
\begin{flalign}
{\text{Level 1. }}&\,\,\,\,\,\,(\alpha_\T,\B_\T)\in\mathbb A[\T]:= \argmin\flp{\norm{u_c-u_{\alpha,\B}}_{L^2(Q)}^2:\,\,(\alpha,\B)\in\mathbb T},\tag{$\mathcal T$-L1}\label{ABtraining_0_1}&\\
{\text{Level 2. }}&\,\,\,\,\,\,u_{\alpha,\B}:=\argmin\flp{\norm{u-u_\eta}_{L^2}^2+\alpha PV_\B(u),\,\, u\in L^1(Q)}.\tag{$\mathcal T$-L2}\label{ABsolution_map}&
\end{flalign}
In particular, for the case that $\alpha=+\infty$, we define
\be\label{infinite_solution}
u_{+\infty}:=\argmin\flp{\norm{u-u_c}_{L^2(Q)}^2:\,\,u\in \mathcal N}\text{ where }\mathcal N:=\operatorname{conv}\fsp{\bigcup_{\B\in\Sigma}\mathcal N(\B)}.
\ee
In \eqref{ABtraining_0_1},  we denote by notation $\mathbb A[\T]$ the collection of optimal solution(s) of scheme $\FT$ with underlying training ground $\T$, and $(\alpha_\T,\B_\T)\in\mathbb A[\T]$ is an optimal solution obtained from training ground $\T$.

\end{define}

We first show that the Level 2 problem \eqref{ABsolution_map} admits a unique solution.
\begin{proposition}\label{unique_exist_lower}
Let $\alpha\in\R^+$ and $\B\in\Sigma$ be given. Then, there exists a unique $u_{\alpha,\B}\in BV_\B(Q)$ such that
\be
u_{\alpha,\B}=\argmin\flp{\norm{u-u_\eta}^2_{L^2(Q)}+\alpha PV_\B(u):\, u\in L^1(Q)}.
\ee
\end{proposition}
\begin{proof}
The proof can be obtained by Proposition \ref{thm_asymptitic_PGV} and the fact that $PV_\B$ is convex.
\end{proof}
\begin{theorem}\label{main_thm}
Let the training ground $\mathbb T$ be given. Then the training scheme $\mathcal T$ admits at least one solution $(\alpha_\T, \B_\T)\in\mathbb T$, and provides an associated optimally reconstructed image $u_{\alpha_\T, \B_\T}\in BV_{ \B_\T}(Q)$.
\end{theorem}
\begin{proof}
Let $\seqn{\alpha_n,\B_n}\subset \mathbb T$ be a minimizing sequence obtained from \eqref{ABtraining_0_1}. Then, by the boundedness and closedness of $\Sigma$ in $\ell^{\infty}$, up to a subsequence (not relabeled), there exists $(\alpha_\T,\B_\T)\in \operatorname{cl}(\R^+)\times \Sigma$ such that $\alpha_n\to\alpha_\T$ in $\R$, $\B_n\to \B_\T$ in $\ell^{\infty}$, and 
\be\label{goes_to_hell_heaven}
\limn \,\norm{u_c-u_{\alpha_n,\B_n}}_{L^2(Q)}^2\to m:=\inf\flp{\norm{u_c-u_{\alpha,\B}}_{L^2(Q)}^2:\,\,(\alpha,\B)\in\mathbb T}.
\ee
We divide our arguments into three cases.\\\\
\underline{Case 1:} Assume $\alpha_\T>0$. Then, in view of Theorem \ref{thm_Gamma_conv} and the properties of $\Gamma$-convergence, we have 
\be\label{main_thm_eq}
u_{\alpha_n,\B_n}\wto u_{\alpha_\T,\B_\T}\text{ weakly}\text{ in }L^2(Q),
\ee
where $u_{\alpha_n,\B_n}$ and $u_{\alpha_\T, \B_\T}$ are obtained from \eqref{ABsolution_map}. Thus, we deduce that 
\be
\norm{u_{\alpha_\T, \B_\T}-u_c}_{L^2(Q)}\leq \liminfn \norm{u_{\alpha_n,\B_n}-u_c}_{L^2(Q)}=m,
\ee
which completes the thesis.\\\\
\underline{Case 2:} Assume $\alpha_\T=0$. Then by \eqref{goes_to_hell_heaven}, up to a subsequence, there exists $\bar u\in L^2(Q)$ such that $u_{\alpha_n,\B_n}\wto \bar u$ weakly in $L^2$. We claim that $u_{\alpha_n,\B_n}\to u_\eta$ in $L^2$ strong. Extend $u_\eta$ by zero outside $Q$ and we define 
\be
u_{\eta}^\e:=u_\eta\ast \eta_\e
\ee
where $\eta_\e$ is the standard mollifier. Then we have $u_\eta^\e\in C_c^\infty(\rn)$ and $u_{\eta}^\e\to u_\eta$ strongly in $L^2(\rn)$. By the optimality condition of \eqref{ABsolution_map}, we have
\begin{align}
&\norm{u_{\alpha_n,\B_n}-u_\eta}^2_{L^2(Q)}+\alpha_n PV_{\B_n}(u_{\alpha_n,\B_n})\\
&\leq \norm{u_{\eta}^\e-u_\eta}^2_{L^2(Q)}+\alpha_nPV_{\B_n}(u_\eta^\e)\leq \norm{u_{\eta}^\e-u_\eta}^2_{L^2(Q)}+\alpha_nN^d \norm{u_\eta^\e}_{W^{d,1}(\rn)}.
\end{align}
That is, we have
\be
\norm{u_{\alpha_n,\B_n}-u_\eta}^2_{L^2(Q)}\leq  \norm{u_{\eta}^\e-u_\eta}^2_{L^2(Q)}+\alpha_nN^d \norm{u_\eta^\e}_{W^{d,1}(\rn)},
\ee
and we are done by letting $\alpha_n\to 0$ first and $\e\to 0$ second.\\\\
\underline{Case 3:} Assume $\alpha_\T=+\infty$. Reasoning as in Case 2, we have again that there exists $\bar u\in L^2(Q)$ such that 
\be
u_{\alpha_n,\B_n}\wto \bar u\text{ and }PV_{\B_\T}(\bar u)=0.
\ee
Then, in this case we have, by \eqref{infinite_solution}, that
\be
m=\liminfn\norm{u_{\alpha_n,\B_n}-u_c}_{L^2(Q)}\geq \norm{\bar u-u_c}_{L^\infty(Q)}\geq \norm{u_{+\infty}-u_c}_{L^2(Q)},
\ee
as desired.
\end{proof}
\subsection{Numerical realization and finite approximation of scheme $\mathcal T$}\label{sec_numerical_finiteapprox}
For the numerical realization of training scheme $\mathcal T$, we in addition require that the training ground $\T$ satisfies the following assumption.
\begin{assumption}\label{kappa_training_ground}
Let the order $d\in\N$ be given.
\begin{enumerate}[1.]
\item
We assume the intensity parameter $\alpha$ satisfies the \emph{box-constraint} (see, e.g. \cite{bergounioux1998optimal,de2013image}). That is, there exists a constant $P\in\R^+$, chosen by user, such that $\alpha\in[0,P]$.
\item
We assume the collection $\Sigma$ of operator $\B$ satisfies the following two conditions.
\begin{enumerate}[a.]
\item
Each operator $\B\in\Sigma$ has at most order $d$ (the box-constraint on order of $\B$);
\item
For any $\B_1$, $\B_2\in \Sigma$, the continuity assumption 
\be\label{cont_sigma_PV1}
\abs{PV_{\B_1}(u)-PV_{\B_2}(u)}\leq O\fsp{\abs{\B_1-\B_2}_{\ell^\infty}}PV_{\B_1}(u)
\ee
and
\be\label{cont_sigma_PV2}
\abs{PV_{\B_1}(u)-PV_{\B_2}(u)}\leq O\fsp{\abs{\B_1-\B_2}_{\ell^\infty}}PV_{\B_2}(u)
\ee
holds.
\end{enumerate}
\end{enumerate}
\end{assumption}

The following corollary is a direct consequence of Theorem \ref{main_thm}
\begin{corollary}
The training scheme $\mathcal T$, with a underlying training ground $\T$ satisfies Assumption \eqref{kappa_training_ground}, admits at least one solution $(\alpha_\T,\B_\T)\in\mathbb T$, and provides an associated optimally reconstructed image $u_{\alpha_\T,\B_\T}\in BV_{\B_\T}(Q)$.
\end{corollary}
\begin{proof}
The argument is identical to the argument used in Theorem \ref{main_thm}, Case 1 \& Case 2.
\end{proof}
Recall the definition of the \emph{assessment operator} from \eqref{cost_map_intro} that
\be
\mathcal A(\alpha,\B):=\norm{u_c-u_{\alpha,\B}}_{L^2(Q)}^2,\text{ for }(\alpha,\B)\in\T.
\ee 
As discussed in Section \ref{sec_introduction}, the Level 1 problem \eqref{ABtraining_0_1} for scheme $\mathcal T$ is equivalent to find global minimizers of $\mathcal A(\alpha,\B)$ among the training ground $\mathbb T$. However, in view of the counter-example provided in \cite{carolapani2018bilevel}, the assessment function $\mathcal A(\cdot)$ is not convex, and hence the traditional methods like \emph{Newton's descent} or \emph{Line search} could trapped into local minimums, but not convergence to global minimums.\\\\
We overcome this problem by using a finite approximation method original introduced in \cite{carolapani2018bilevel}. Recall the constant $P>0$ given in box-constraint stated in Assumption \ref{kappa_training_ground}.
\begin{define}[The \emph{Finite TrainingGround} and \emph{Finite Grid}]\label{def_BCFTG}
 Let $l\in\mathbb N$ be given.
\begin{enumerate}[1.]
\item
We define the \emph{step size} $\delta_l$ by 
 \be
 \delta_l := P/l;
 \ee
 \item
we define the finite set $\T_l[P]\subset[0,P]$ via
\be
\T_l[P]:=\flp{0,\,\,\delta_l,\,\,2\delta_l,\,\,\ldots,\,\,i\delta_l,\ldots, P}.
\ee
\item
we define the finite set $\T_l[\Sigma]\subset \Sigma$ via
\be
\mathbb T_l[\Sigma]:=\bigcup_{k\geq 1}  T_k[\Sigma]
\ee
where each $T_k[\Sigma]$ is a singleton contains one operator $\B\in\Sigma$ and defined recursively in the following steps. 
\begin{enumerate}[Step 1.] 
\item
Define
\be
\B_0\in\argmin\flp{\norm{\B}:\,\,\B\in\Sigma},\,\,T_0[\Sigma]=\flp{\B_0},\text{ and }\Sigma_0:=\Sigma.
\ee
We also denote $Q_l[\B_0]\subset \Sigma_0$ be the cube centered at $\B_0$ with side length $\Delta_l$. 
\item
Define
\be
\Sigma_1:=\Sigma_{0}\setminus Q_l[\B_{0}],\,\,\B_1\in\argmin\flp{\abs{\B}:\,\,\B\in\Sigma_1},
\ee
and
\be
T_1[\Sigma]:=\flp{\B_1};
\ee
\item[$\vdots$]
\item[Step $j$.]
Define
\be
\Sigma_k:=\Sigma_{k-1}\setminus Q_l[\B_{k-1}],\,\,\B_j\in\argmin\flp{\norm{\B}:\,\,\B\in\Sigma_j},
\ee
and
\be
T_k[\Sigma]:=\flp{\B_j}.
\ee
Repeat until $\Sigma_k=\emptyset$.
\end{enumerate}
\item
we define the \emph{Finite TrainingGround} $\mathbb T_l$ at step $l\in\N$ by
\be
\mathbb T_l:=\mathbb T_l[\alpha]\times \mathbb T_l[\Sigma].
\ee
\item
for $i$, $j\in\N$, we define the $(i,j)$-th \emph{FiniteGrid} at step $l$ by 
\be\label{eq_finite_grid_one}
\mathbb G_l(i,j):=[i\Delta_l,(i+1)\Delta_l]\times Q_l[\B_j].
\ee
\end{enumerate}
\end{define}
\begin{remark}
From the definition of $\Sigma_k$ and $Q_l$, we have, for any $l\in\N$ fixed, there exists an upper bound $M\in\N$, depends on $l$, such that $\Sigma_{M}=\emptyset$. In another word, we have $\mathcal H^0\fsp{\mathbb T_l[\B]}<+\infty$ and hence
\be
\mathcal H^0\fsp{\mathbb T_l}<+\infty,\text{ for each }l\in\N\text{ fixed.}
\ee
Then, the optimal parameters of scheme $\mathcal T$ (global minimizers of $\mathcal A(\cdot,\cdot)$) over finite training ground $\mathbb T_l$
\be
(\alpha_{\T_l},\B_{\T_l})\in\mathbb A[\T_l] :=\argmin\flp{\norm{u_{\alpha,\B}-u_c}_{L^2(Q)}^2:\,\,(\alpha,\B)\in \mathbb T_l},
\ee
can be determined exactly by evaluating $\mathcal A(\cdot)$ over each elements of $\mathbb T_l$.
\end{remark}
The main result of Section \ref{sec_numerical_finiteapprox} reads as follows.
\begin{theorem}[finite approximation and error estimation]\label{PV_finiteapprox_result}
Let a training ground $\mathbb T$ satisfies Assumption \ref{kappa_training_ground} be given, and $\T_l\subset \T$ be constructed as in Definition \ref{def_BCFTG}. Then the following assertions hold:
\begin{enumerate}[1.]
\item
as $l\to\infty$, we have
\be\label{tgv_soln_error}
\operatorname{dist}(\mathbb A[\T],\mathbb A[\T_l])\to 0.
\ee
\item
let $\delta>0$ be given. Then for each $l\in\N$ we have the following estimation hold
\be\label{tgv_cost_error}
{\CC(\alpha_{\T_l},\B_{\T_l})-\CC(\alpha_\T,\B_\T)}\leq 4KP \fmp{O\fsp{P/l}+1/l}^{1/2}\norm{u_\eta}_{W^{d,1}(Q)}^{1/2}/\delta^d+\delta/2,
\ee
for any $(\alpha_\T,\B_\T)\in\mathbb A[\T]$ and $(\alpha_{\T_l},\B_{\T_l})\in\mathbb A[\T_l]$.
\end{enumerate}
\end{theorem}
We sub-divide our argument into Section \ref{fixed_B_operator} and Section \ref{fixed_alpha_operator}, in which we discuss the properties of reconstructed image $u_{\alpha,\B}$ with $\B\in\Sigma$ fixed and $\alpha\in\R^+$ fixed, respectively.
\subsubsection{Properties of reconstructed image $u_{\alpha,\B}$ with respect to $\alpha\in\R^+$}\label{fixed_B_operator}
Since $\B\in\Sigma$ is fixed, we abbreviate $u_{\alpha,\B}$ and $PV_\B$ by $u_\alpha$ and $PV$, respectively,  in Section \ref{fixed_B_operator}. 
\begin{proposition}\label{g_alpha_auxiliary}
We collect two auxiliary results in this proposition.
\begin{enumerate}[1.]
\item
The function $g(\alpha):=PV(u_\alpha)$ is continuous decreasing;
\item
Assume in addition that 
\be\label{eq_g_alpha_auxiliary1}
PV(u_\eta)>PV(u_c).
\ee
Then, there exists $\alpha>0$ such that 
\be\label{eq_g_alpha_auxiliary2}
\norm{u_\alpha-u_c}_{L^2(Q)}< \norm{u_\eta-u_c}.
\ee
\end{enumerate} 
\end{proposition}
\begin{proof}
We show Assertion 1 first. The continuity of $g(\alpha)$ can be deduced from Theorem \ref{thm_Gamma_conv}. Next, let $0\leq\alpha_1<\alpha_2<+\infty$ be given, we observe, from the optimality condition of \eqref{ABsolution_map}, that
\be
\norm{u_{\alpha_1}-{u_\eta}}_{L^2(Q)}^2+\alpha_1 PV\fsp{u_{\alpha_1}}\leq \norm{u_{\alpha_2}-u_{\eta}}_{L^2(Q)}^2+\alpha_1 PV\fsp{u_{\alpha_2}}
\ee
and 
\be
\norm{u_{\alpha_2}-u_\eta}_{L^2(Q)}^2+\alpha_2 PV\fsp{u_{\alpha_2}}\leq \norm{u_{\alpha_1}-u_\eta}_{L^2(Q)}^2+\alpha_2 PV\fsp{u_{\alpha_1}}.
\ee
Adding up the previous two inequalities yields
\be
\alpha_1 PV\fsp{u_{\alpha_1}}+\alpha_2 PV\fsp{u_{\alpha_2}}\leq \alpha_1 PV\fsp{u_{\alpha_2}}+\alpha_2 PV\fsp{u_{\alpha_1}},
\ee
which implies that $PV\fsp{u_{\alpha_2}}\leq  PV\fsp{u_{\alpha_1}}$ as desired. \\\\
Now we claim Assertion 2. From Theorem \ref{main_sub_diff_pvb}, we have $\partial PV(u_\alpha)$, the sub-differential of $PV$ at $u_{\alpha}$, is well defined. We observe, for any $\alpha>0$, that
\begin{align}
&\norm{u_\eta - u_c}_{L^2(Q)}^2-\norm{u_{\alpha} - u_c}_{L^2(Q)}^2 \\
&= 2\fjp{u_\eta-u_{\alpha},u_{\alpha}-u_c}+\norm{u_\eta-u_{\alpha}}_{L^2(Q)}^2
= 2\alpha\fjp{\partial PV(u_\alpha),u_{\alpha}-u_c}+\norm{u_\eta-u_{\alpha}}_{L^2(Q)}^2\\
&= 2\alpha\fjp{\partial PV(u_\alpha),u_{\alpha}}-2\alpha\fjp{\partial PV(u_\alpha),u_c}+\norm{u_\eta-u_{\alpha}}_{L^2(Q)}^2 \\
&\geq 2\alpha\fmp{PV(u_{\alpha}) -PV(u_c)}+\norm{u_\eta-u_{\alpha}}_{L^2(Q)}^2,
\end{align}
where at the last inequality we use the property of sub-differential operator, and we obtain that
\be\label{alpha_l_lb_pre2}
\norm{u_\eta - u_c}_{L^2(Q)}^2-\norm{u_{\alpha} - u_c}_{L^2(Q)}^2\geq 2\alpha\fmp{PV(u_{\alpha}) -PV(u_c)}+\norm{u_\eta-u_{\alpha}}_{L^2(Q)}^2.
\ee
Next, in view of Assertion 1, we have that $PV(u_{\alpha})$ is continuous decreasing and hence, together with \eqref{eq_g_alpha_auxiliary1}, there exists $\bar \alpha>0$ such that 
\be\label{alpha_l_lb_pre}
PV(u_{\bar\alpha}) -PV(u_c)\geq\frac14\fmp{PV(u_\eta)-PV(u_c)}>0.
\ee
Hence, we conclude \eqref{eq_g_alpha_auxiliary2} by combining \eqref{alpha_l_lb_pre2} and \eqref{alpha_l_lb_pre}.
\end{proof}

\begin{proposition}
Let $\alpha_1$ and $\alpha_2\in\R^+$ be given. Then we have that
\be
\norm{u_{\alpha_1}-u_{\alpha_2}}_{L^2(Q)}^2\leq \abs{\alpha_1-\alpha_2} \fsp{PV(u_{\alpha_1})+PV(u_{\alpha_2})}
\ee
as desired.
\end{proposition}
\begin{proof}
Without lose of generality we assume that $\alpha_1<\alpha_2$. In view of Theorem \ref{main_sub_diff_pvb}, and from the optimality condition of \eqref{ABsolution_map} we have 
\be
u_{\alpha_1}-u_\eta=-\alpha_1 \partial PV(u_{\alpha_1})\text{ and }u_{\alpha_2}-u_\eta=-\alpha_2\partial PV(u_{\alpha_2}).
\ee
Subtracting one from another and multiplying with $u_{\alpha_1}-u_{\alpha_2}$ and integration over $Q$, we obtain that
\begin{align}\label{semi_subdiff_one}
&\norm{u_{\alpha_1}-u_{\alpha_2}}_{L^2(Q)}^2\\
&=\alpha_1\fjp{\partial PV(u_{\alpha_2})-\partial PV(u_{\alpha_1}),u_{\alpha_1}-u_{\alpha_2}}+(\alpha_2-\alpha_1)\fjp{\partial PV(u_{\alpha_2}),u_{\alpha_1}-u_{\alpha_2}}.
\end{align}
Since the seminorm $PV$ is proper, $l.s.c.$, and convex, we have $\partial PV$ is a monotone maximal operator and hence
\be
\fjp{\partial PV(u_{\alpha_2})-\partial PV(u_{\alpha_1}),u_{\alpha_2}-u_{\alpha_1}}\geq 0.
\ee
This, together with \eqref{semi_subdiff_one} and Assertion 1 from Proposition \ref{g_alpha_auxiliary}, we obtain that 
\begin{align}
&\norm{u_{\alpha_1}-u_{\alpha_2}}_{L^2(Q)}^2\leq (\alpha_2-\alpha_1)\fjp{\partial PV(u_{\alpha_2}),u_{\alpha_1}-u_{\alpha_2}}\\
&\leq (\alpha_2-\alpha_1)PV(u_{\alpha_1}-u_{\alpha_2})\leq (\alpha_2-\alpha_1)\fmp{PV(u_{\alpha_1})+PV(u_{\alpha_2})},
\end{align}
where at the second last inequality we used Remark \ref{rmk_test_subgradient}, and hence the thesis.
\end{proof}
\subsubsection{Properties of reconstructed image $u_{\alpha,\B}$ with respect to $\B\in\Sigma$}\label{fixed_alpha_operator}
Analogously to Section  \ref{fixed_B_operator}, in Section \ref{fixed_alpha_operator} we abbreviate $u_{\alpha,\B}$ by $u_\B$,  for $\alpha\in\R^+$ fixed. Recall the structure of $\B$ from Notation \ref{operator_B}.\\\\
Moreover, in Section \ref{fixed_alpha_operator}, we further restrict the corrupted image $u_\eta\in L^2(Q)$ satisfies that there exists $0<M_1<M_2<+\infty$ such that 
\be\label{restricted_noisy_corrupted}
0<M_1\leq u_\eta(x)\leq M_2<+\infty,\text{ for a.e. }x\in Q.
\ee
In this way, we have that the reconstructed image
\be
u_\B=\argmin\flp{\norm{u-u_\eta}_{L^2(Q)}^2+PV_\B(u):\,\, u\in L^1(Q)}
\ee
also satisfies that 
\be\label{restricted_reconstruction}
M_1\leq u_\B(x)\leq M_2,\text{ for a.e. }x\in Q.
\ee

Before we move to next proposition, we call the following result regarding the Lebesgue point.
\begin{theorem}[Lebesgue-besicovitch differentiation theorem]\label{lebesgue_point_thm}
Let $\mu$ be a Radon measure on $\rn$ and $f\in L^1_\loc(\rn,\mu)$. Then
\be
\lim_{r\to 0}\fint_{B(x,r)}f\,d\mu = f(x)
\ee
for $\mu$ a.e. $x\in\rn$.
\end{theorem}

\begin{proposition}\label{thm_estimation_pvb_l2}
Let $u_\eta\in L^2(Q)$ satisfies \eqref{restricted_noisy_corrupted} be given. Let $\B_1$ and $\B_2\in \Sigma$ be given. Then we have that 
\be
\norm{u_{\B_1}-u_{\B_2}}^2_{L^2(Q)}\leq O\fsp{\abs{\B_1- \B_2}_{\ell^\infty}}\fmp{PV_{\B_1}(u_{\B_1})+PV_{\B_1}(u_{\B_2})},
\ee
where $u_{\B}$ is defined in \eqref{ABsolution_map}.
\end{proposition}
\begin{proof}
By Theorem \ref{main_sub_diff_pvb}, we have the sub-differential $\partial PV_{\B_1}$ and $\partial PV_{\B_2}$ are well defined. Then, by the optimality condition of \eqref{ABsolution_map} we have that
\be
u_{\B_1}-u_\eta=-\partial PV_{\B_1}(u_{\B_1})\text{ and }u_{\B_2}-u_\eta=-\partial PV_{\B_2}(u_{\B_2}),
\ee
Subtracting with one from another, we have that
\begin{align}
&u_{\B_1}-u_{\B_2} = \partial PV_{\B_2}(u_{\B_2})-\partial PV_{\B_1}(u_{\B_1}) \\
&= \partial PV_{\B_2}(u_{\B_2})-\partial PV_{\B_2}(u_{\B_1})+\partial PV_{\B_2}(u_{\B_1})- \partial PV_{\B_1}(u_{\B_1}).
\end{align}
Multiplying both side by $u_{\B_1}-u_{\B_2}$ and integrate over $Q$, we obtain that
\begin{align}\label{p_var_est_1}
\norm{u_{\B_2}-u_{\B_1}}_{L^2(Q)}^2=&-\fjp{\partial PV_{\B_2}(u_{\B_2})-\partial PV_{\B_2}(u_{\B_1}),u_{\B_2}-u_{\B_1}}\\
&+\fjp{\partial PV_{\B_2}(u_{\B_1})- \partial PV_{\B_1}(u_{\B_1}),u_{\B_1}-u_{\B_2}}.
\end{align}
Since $PV_\B$ is convex, $\partial PV_\B$ is a monotone maximal operator. Therefore, we have
\be\label{monotone_sub_grad}
\fjp{\partial PV_{\B_2}(u_{\B_2})-\partial PV_{\B_2}(u_{\B_1}),u_{\B_2}-u_{\B_1}}\geq 0.
\ee
We next estimate the second part of \eqref{p_var_est_1}. Firstly, from the definition of sub-gradient, we have
\be\label{l2_beta_estimation3}
\fjp{\partial PV_{\B_2}(u_{\B_1})-\partial PV_{\B_1}(u_{\B_1}),u_{\B_1}} = PV_{\B_2}(u_{\B_1})-PV_{\B_1}(u_{\B_1})\leq cPV_{\B_1}(u_{\B_1}).
\ee 
where 
\be
c:=O\fsp{\abs{\B_1-\B_2}_{\ell^\infty}}
\ee
is the constant used in \eqref{cont_sigma_PV1}.  Moreover, from \eqref{cont_sigma_PV1} we also deduce that 
\be\label{eq_c_difference}
-c\abs{\B_2u_{\B_1}}\leq \abs{\B_2u_{\B_1}}(V)-\abs{\B_1u_{\B_1}}(V)\leq c\abs{\B_2u_{\B_1}}.
\ee
Next, Let $v_{\B_1}$ and $v_{\B_2}$ be obtained from Proposition \ref{main_sub_diff_pvb} as the sub-differential of $PV_{\B_1}(u_{\B_1})$ and $PV_{\B_2}(u_{\B_1})$, respectively. Then, by Proposition \ref{nested_sub_gradient}, for any open set $V\subset Q$ we have that
\be
\abs{\B_2u_{\B_1}}(V)=\int_V u_{\B_1}[\B_2^\ast v_{\B_2}]\,dx\text{ and }\abs{\B_1u_{\B_1}}(V)=\int_V u_{\B_1}[\B_1^\ast v_{\B_1}]\,dx.
\ee
This, and together with \eqref{eq_c_difference}, we conclude
\be
-c\int_V u_{\B_1}[\B_2^\ast v_{\B_2}]\,dx\leq \int_V u_{\B_1}[\B_2^\ast v_{\B_2}]\,dx-\int_V u_{\B_1}[\B_1^\ast v_{\B_1}]\,dx
\leq c\int_V u_{\B_1}[\B_1^\ast v_{\B_1}]\,dx.
\ee
Thus, we could further write, by taking $Q(x,\delta):=[x-\delta,x+\delta]^N$, a cube centered at $x$ with side length $2\delta$, that 
\begin{align}
-c \fint_{Q(x,\delta)} u_{\B_1}[\B_2^\ast v_{\B_2}]\,dx&\leq \fint_{Q(x,\delta)}\fsp{ u_{\B_1}[\B_2^\ast v_{\B_2}]-u_{\B_1}[\B_1^\ast v_{\B_1}]}dx\\
& =\fint_{Q(x,\delta)} u_{\B_1}[\B_2^\ast v_{\B_2}]\,dx-\fint_{Q(x,\delta)} u_{\B_1}[\B_1^\ast v_{\B_1}]\,dx\\
&\leq c \fint_{Q(x,\delta)} u_{\B_1}[\B_1^\ast v_{\B_1}]\,dx
\end{align}
By Assertion 2, Theorem \ref{main_sub_diff_pvb}, we have $\B_1^\ast v_{\B_1}\in L^1(Q)$. Since $u_{\B_1}\in L^\infty(Q)$, we have $u_{\B_1}\B_1^\ast v_{\B_1}\in L^1(Q)$. Thus, we could apply the Lebesgue point in  Theorem \ref{lebesgue_point_thm} and take $\delta\to 0$ to conclude that 
\be
-c u_{\B_1}\fmp{\B_2^\ast v_{\B_2}}\leq u_{\B_1}[\B_2^\ast v_{\B_2}]-u_{\B_1}[\B_1^\ast v_{\B_1}]\leq cu_{\B_1}[\B_1^\ast v_{\B_1}],
\ee
for a.e. $x\in Q$. That is, we have
\be
-c u_{\B_1}\fmp{\B_2^\ast v_{\B_2}}\leq u_{\B_1}[\B_2^\ast v_{\B_2}-\B_1^\ast v_{\B_1}]\leq c u_{\B_1}\fmp{\B_1^\ast v_{\B_1}}, 
\ee
and together with the fact that $u_{\B_1}\geq1$ (see \eqref{restricted_reconstruction}), we deduce that 
\be\label{sub_gradient_xiangaixiangsha}
-c\fmp{\B_2^\ast v_{\B_2}}\leq [\B_2^\ast v_{\B_2}-\B_1^\ast v_{\B_1}]\leq c \fmp{\B_1^\ast v_{\B_1}},
\ee
for a.e. $x\in Q$.\\\\
On the other hand, again by \eqref{restricted_reconstruction}, we have that $-u_{\B_1}+2M_2>1$, and hence
\begin{align}\label{simi_subgradient_estimation1}
&\fjp{\partial PV_{\B_2}(u_{\B_1})-\partial PV_{\B_1}(u_{\B_1}),-u_{\B_2}}\\
&=\fjp{\partial PV_{\B_2}(u_{\B_1})-\partial PV_{\B_1}(u_{\B_1}),-u_{\B_2}+2M_2-2M_2}\\
&= \fjp{\partial PV_{\B_2}(u_{\B_1})-\partial PV_{\B_1}(u_{\B_1}),-u_{\B_2}+2M_2}\\
&\,\,\,\,\,\,\,\,\,\,+\fjp{\partial PV_{\B_2}(u_{\B_1})-\partial PV_{\B_1}(u_{\B_1}),-2M_2}\\
&\leq c \fjp{\fmp{\B_1^\ast v_{\B_1}},-u_{\B_2}+2M_2}+\fjp{\partial PV_{\B_2}(u_{\B_1})-\partial PV_{\B_1}(u_{\B_1}),-2M_2}.
\end{align}
Note that, as $v_{\B_1}\in W_0^p\fmp{\B_1}(Q)$, by Remark \ref{rmk_test_subgradient} we observe that
\be\label{simi_subgradient_estimation2}
\fjp{\B_1^\ast v_{\B_1},-u_{\B_2}+2M_2}\leq PV_{\B_1}(-u_{\B_2}+2M_2)=PV_{\B_1}(u_{\B_2}),
\ee
and, since the constants belongs to the kernel of $PV_{\B_2}$, 
\be\label{simi_subgradient_estimation3}
\fjp{\partial PV_{\B_2}(u_{\B_1})-\partial PV_{\B_1}(u_{\B_1}),-2M_2}=0.
\ee
Therefore, by combing \eqref{simi_subgradient_estimation1}, \eqref{simi_subgradient_estimation2}, and \eqref{simi_subgradient_estimation3}, we obtain that 
\be
\fjp{\partial PV_{\B_2}(u_{\B_1})-\partial PV_{\B_1}(u_{\B_1}),-u_{\B_2}}\leq c PV_{\B_1}(u_{\B_2}).
\ee
This, together with \eqref{p_var_est_1}, \eqref{monotone_sub_grad}, and \eqref{l2_beta_estimation3}, we conclude our thesis.
\end{proof}

\subsubsection{$L^2$-distance estimation of reconstructed image $u_{\alpha,\B}$}
We start with a relaxation result regarding to the corrupted image $u_\eta$.
\begin{proposition}\label{prop_approx_fiderly}
Let $u_\eta\in L^2(Q)$ be given. Let $\seqe{u_{\eta}^\e}\subset L^2(Q)$ such that $u_{\eta}^\e\to u_\eta$ strongly in $L^2$. For arbitrary $(\alpha,\B)\in\T$, define
\be\label{approx_fiderlity}
u^\e_{\alpha,\B}:=\argmin\flp{\norm{u-u_{\eta}^\e}_{L^2}^2+\alpha PV_\B(u):\,\, u\in L^1(Q)}.
\ee
Then we have 
\be\label{approx_fider_xiao1}
\norm{u_{\alpha,\B}-u_{\alpha,\B}^\e}_{L^2(Q)}\leq \norm{u_\eta^\e-u_\eta}_{L^2(Q)}
\ee
and 
\be\label{approx_fider_xiao2}
\lime PV_\B(u_{\alpha,\B}^\e)= PV_\B(u_{\alpha,\B}),
\ee
where $u_{\alpha,\B}$ is defined in \eqref{ABsolution_map}.
\end{proposition}
\begin{proof}
From the optimality condition of \eqref{approx_fiderlity} and \eqref{ABsolution_map}, we have 
\be
u_{\alpha,\B}-u_{\alpha,\B}^\e+u_\eta^\e-u_\e = \alpha\partial PV_\B(u_{\alpha,\B}^\e) -\alpha\partial PV_\B(u_{\alpha,\B}).
\ee
Multiplying $u_{\alpha,\B}-u_{\alpha,\B}^\e$ on the both hand side, we have
\begin{align}
&\norm{u_{\alpha,\B}-u_{\alpha,\B}^\e}_{L^2(Q)}^2+\fjp{u_\eta^\e-u_\e,u_{\alpha,\B}-u_{\alpha,\B}^\e} \\
&= \alpha\fjp{\partial PV_\B(u_{\alpha,\B}^\e) -\partial PV_\B(u_{\alpha,\B}),u_{\alpha,\B}-u_{\alpha,\B}^\e}\leq 0,
\end{align}
where at the last inequality we used the fact that $\partial PV_\B$ is a maximal monotone operator, and we conclude \eqref{approx_fider_xiao1} as desired.\\\\
We next claim \eqref{approx_fider_xiao2}. We assume that $\alpha\in\R^+$, otherwise there is nothing to prove. By \eqref{approx_fider_xiao1}, we have that
\be\label{l2strong_app_fidely}
u_{\alpha,\B}\to u_{\alpha,\B}^\e\text{ strongly in }L^2.
\ee
This, and together with Proposition \ref{lower_semicontity}, we deduce that 
\be\label{lower_semicontity_app_fid}
\liminf_{\e\to 0} PV_\B(u^\e_{\alpha,\B})\geq PV_\B(u_{\alpha,\B}).
\ee
On the other hand, in view of the optimality condition of \eqref{approx_fiderlity} again, we have
\be
\norm{u^\e_{\alpha,\B}-u_{\eta}^\e}_{L^2}^2+\alpha PV_\B(u^\e_{\alpha,\B})\leq \norm{u_{\alpha,\B}-u_{\eta}^\e}_{L^2}^2+\alpha PV_\B(u_{\alpha,\B}),
\ee
or
\be
\alpha PV_\B(u^\e_{\alpha,\B})\leq \norm{u_{\alpha,\B}-u_{\eta}^\e}_{L^2}^2-\norm{u^\e_{\alpha,\B}-u_{\eta}^\e}_{L^2}^2+\alpha PV_\B(u_{\alpha,\B}).
\ee
Hence, by \eqref{l2strong_app_fidely}, we have that 
\begin{align}
&\limsup_{\e\to 0}\alpha PV_\B(u^\e_{\alpha,\B})\\
&\leq \limsup_{\e\to 0}\fmp{\norm{u_{\alpha,\B}-u_{\eta}^\e}_{L^2}^2-\norm{u^\e_{\alpha,\B}-u_{\eta}^\e}_{L^2}^2}+ \alpha PV_\B(u_{\alpha,\B})\\
&=\alpha PV_\B(u_{\alpha,\B}).
\end{align}
This, and \eqref{lower_semicontity_app_fid}, allows us to conclude \eqref{approx_fider_xiao2} as desired.
\end{proof}
We next present an improved version of Proposition \ref{thm_estimation_pvb_l2}, in which we remove the assumption that $u_\eta$ need to satisfy  the boundness assumption \eqref{restricted_noisy_corrupted}.
\begin{corollary}\label{coro_make_in_alpha}
Let $u_\eta\in L^2(Q)$, $\alpha\in\R^+$, and $\B_1$, $\B_2\in\Sigma$ be given. Then the following estimation holds.
\be
\norm{u_{\alpha,\B_1}-u_{\alpha,\B_2}}^2_{L^2(Q)}\leq \alpha \cdot O\fsp{\abs{\B_1- \B_2}_{\ell^\infty}}\fmp{PV_{\B_1}(u_{\alpha,\B_1})+PV_{\B_1}(u_{\alpha,\B_2})},
\ee
where $u_{\alpha,\B}$ is defined in \eqref{ABsolution_map}.
\end{corollary}
\begin{proof}
Let $M\in\N$ be given, and define 
\be
u_\eta^M:=-M\wedge u_\eta\vee M.
\ee
Also, we define that
\be\label{opti_K_KK1}
u_{\alpha,\B}^M:=\argmin\flp{\norm{u-u_{\eta}^M}_{L^2}^2+\alpha PV_\B(u):\,\, u\in L^1(Q)}
\ee
and 
\be\label{opti_K_KK2}
\bar u_{\alpha,\B}^M:=\argmin\flp{\norm{u-(u_{\eta}^M+2M)}_{L^2}^2+\alpha PV_\B(u):\,\, u\in L^1(Q)}.
\ee
We claim that 
\be\label{constant_shifting}
\bar u_{\alpha,\B}^M = u_{\alpha,\B}^M+2M.
\ee
We observe that
\begin{align}
&\norm{\bar u_{\alpha,\B}^M-\fsp{u_{\eta}^M+2M}}_{L^2}^2+\alpha PV_\B\fsp{\bar u_{\alpha,\B}^M}\\
&\leq \norm{ u_{\alpha,\B}^M+2M-\fsp{u_{\eta}^M+2M}}_{L^2}^2+\alpha PV_\B\fsp{ u_{\alpha,\B}^M+2M}\\ &=  \norm{ u_{\alpha,\B}^M-u_{\eta}^M}_{L^2}^2+\alpha PV_\B\fsp{ u_{\alpha,\B}^M}\\
&\leq \norm{\bar u_{\alpha,\B}^M-2M-u_{\eta}^M}_{L^2}^2+\alpha PV_\B\fsp{\bar u_{\alpha,\B}^M-2M}\\
&=\norm{\bar u_{\alpha,\B}^M-\fsp{u_{\eta}^M+2M}}_{L^2}^2+\alpha PV_\B\fsp{\bar u_{\alpha,\B}^M},
\end{align}
where at the first inequality we used the optimality condition on \eqref{opti_K_KK2}, and at the last inequality we used the optimality condition on \eqref{opti_K_KK1}. Thus, we have
\begin{align}
&\norm{\bar u_{\alpha,\B}^M-(u_{\eta}^M+2M)}_{L^2}^2+\alpha PV_\B(\bar u_{\alpha,\B}^M)\\
&= \norm{ u_{\alpha,\B}^M+2M-(u_{\eta}^M+2M)}_{L^2}^2+\alpha PV_\B( u_{\alpha,\B}^M+2M),
\end{align}
and we conclude \eqref{constant_shifting} in view of the uniqueness of the minimizer. Thus, we have 
\be
\norm{u^M_{\alpha,\B_1}-u^M_{\alpha,\B_2}}_{L^2(Q)} = \norm{\bar u^M_{\alpha,\B_1}-\bar u^M_{\alpha,\B_2}}_{L^2(Q)}.
\ee
Therefore, we could assume that, without lose of generality, $u_\eta^M\geq M>0$. In another word, we have $u_\eta^M$ satisfies \eqref{restricted_noisy_corrupted}.\\\\
Next, by the optimality condition of \eqref{opti_K_KK1}, we have that
\be
\frac1\alpha(u^M_{\alpha,\B_1}-u_\eta)=- \partial PV_{\B_1}(u^M_{\alpha,\B_1})\text{ and }\frac1\alpha(u^M_{\alpha,\B_2}-u_\eta)=-\partial PV_{\B_2}(u^M_{\alpha,\B_2}).
\ee
Following exactly the same argument used in Proposition \ref{thm_estimation_pvb_l2} (in \eqref{simi_subgradient_estimation1} we use $2M$ instead of $M_2$), we obtain that 
\be
\frac1\alpha\norm{u^M_{\alpha,\B_1}-u^M_{\alpha,\B_2}}^2_{L^2(Q)}\leq O\fsp{\abs{\B_1- \B_2}_{\ell^\infty}}\fmp{PV_{\B_1}(u^M_{\alpha,\B_1})+PV_{\B_2}(u^M_{\alpha,\B_1})}.
\ee
In the end, we compute that 
\begin{align}
&\frac1\alpha\norm{u_{\alpha,\B_1}-u_{\alpha,\B_2}}^2_{L^2(Q)}\\
&\leq \frac1\alpha\norm{u^M_{\alpha,\B_1}-u^M_{\alpha,\B_2}}^2_{L^2(Q)} + \frac1\alpha\norm{u_{\alpha,\B_1}-u^M_{\alpha,\B_1}}_{L^2(Q)}^2+\norm{u_{\alpha,\B_2}-u^M_{\alpha,\B_2}}_{L^2(Q)}^2\\
&\leq O\fsp{\abs{\B_1- \B_2}_{\ell^\infty}}\fmp{PV_{\B_1}(u^M_{\alpha,\B_1})+PV_{\B_2}(u^M_{\alpha,\B_1})}\\
&\,\,\,\,\,\,+ \frac1\alpha\norm{u_{\alpha,\B_1}-u^M_{\alpha,\B_1}}_{L^2(Q)}^2+\norm{u_{\alpha,\B_2}-u^M_{\alpha,\B_2}}_{L^2(Q)}^2,
\end{align}
Then, by Proposition \ref{prop_approx_fiderly}, in which $u^\e_\eta$ is replaced by $u^M_\eta$, we conclude our thesis by sending $M\nearrow +\infty$ on the right hand side on the above inequality.
\end{proof}

\begin{proposition}\label{alpha_12_B_12}
Let $(\alpha_1,\B_1)$ and $(\alpha_2,\B_2)\in\mathbb T$ be given. Then we have
\begin{align}
&\norm{u_{\alpha_1,\B_1}-u_{\alpha_2,\B_2}}_{L^2(Q)}^2\\
&\leq 4 \fmp{\alpha_1O\fsp{\abs{\B_1- \B_2}_{\ell^\infty}}+\abs{\alpha_1-\alpha_2}}\fmp{PV_{\B_1}(u_{\alpha_1,\B_1})+PV_{\B_2}(u_{\alpha_1,\B_2})}.
\end{align}
\end{proposition}
\begin{proof}
We compute that 
\begin{align}\label{l2_B12_1st}
&\norm{u_{\alpha_1,\B_1}-u_{\alpha_2,\B_2}}^2_{L^2(Q)}\leq 2\norm{u_{\alpha_1,\B_1}-u_{\alpha_1,\B_2}}^2_{L^2(Q)}+2\norm{u_{\alpha_1,\B_2}-u_{\alpha_2,\B_2}}^2_{L^2(Q)}\\
&\leq 2\alpha_1O\fsp{\abs{\B_1- \B_2}_{\ell^\infty}}\fmp{PV_{\B_1}(u_{\alpha_1,\B_1})+PV_{\B_1}(u_{\alpha_1,\B_2})}+4\abs{\alpha_1-\alpha_2}PV_{\B_2}(u_{\alpha_1,\B_2})\\
&\leq \fmp{2\alpha_1O\fsp{\abs{\B_1- \B_2}_{\ell^\infty}}+4\abs{\alpha_1-\alpha_2}}\fmp{PV_{\B_1}(u_{\alpha_1,\B_1})+PV_{\B_1}(u_{\alpha_1,\B_2})}.
\end{align}
Moreover, from \eqref{cont_sigma_PV2}, we have
\be
\abs{PV_{\B_1}(u_{\alpha_1,\B_2})-PV_{\B_2}(u_{\alpha_1,\B_2})}\leq O\fsp{\abs{\B_1- \B_2}_{\ell^\infty}}PV_{\B_2}(u_{\alpha_1,\B_2}).
\ee
Together with \eqref{l2_B12_1st}, we deduce that
\begin{align}
&\norm{u_{\alpha_1,\B_1}-u_{\alpha_2,\B_2}}^2_{L^2(Q)}\\
&\leq 4 \fmp{2\alpha_1O\fsp{\abs{\B_1- \B_2}_{\ell^\infty}}+\abs{\alpha_1-\alpha_2}}\fmp{PV_{\B_1}(u_{\alpha_1,\B_1})+PV_{\B_2}(u_{\alpha_1,\B_2})}
\end{align}
as desired.
\end{proof}

We close this section by proving Theorem \ref{PV_finiteapprox_result}
\begin{proof}[Proof of Theorem \ref{PV_finiteapprox_result}]
The Assertion 1 is the direct result of Theorem \ref{thm_Gamma_conv}.\\\\
We next claim \eqref{tgv_cost_error}. We assume that $u\in C^\infty(\bar Q)$ for a moment. Indeed, for any $(\alpha, \B)\in \mathbb T$, we could extract a sequence $\seql{(\alpha_l,\B_l)}\subset \mathbb T$, where for each $l\in\N$, $(\alpha_l,\B_l)\in\mathbb T_l$, such that $(\alpha_l,\B_l)\to (\alpha, \B)$. We observe that, by  Proposition \ref{alpha_12_B_12},
\begin{align}\label{two_step_estimation}
&\abs{\CC(\alpha_l, \B_l)- \CC(\alpha, \B)}\\
&= \abs{\norm{u_{\alpha_l, \B_l}-u_c}_{L^2(Q)}-\norm{u_{\alpha, \B}-u_c}_{L^2(Q)}}
\leq \norm{u_{\alpha_l, \B_l}-u_{\alpha, \B}}_{L^2(Q)}\\
&\leq2 \fmp{\alpha_l2O\fsp{\abs{\B_1- \B_2}_{\ell^\infty}}+\abs{\alpha_l-\alpha}}^{1/2}\fmp{PV_{\B_l}(u_{\alpha_l,\B_l})+PV_{\B}(u_{\alpha_l,\B})}^{1/2}\\
&\leq 4K \fmp{\alpha_l2O\fsp{\abs{\B_l- \B}_{\ell^\infty}}+\abs{\alpha_l-\alpha}}^{1/2}\norm{u_\eta}_{W^{d,1}(Q)}^{1/2}.
\end{align}
Next, take any optimal solution $(\alpha_\T,\B_\T)$ from \eqref{tgv_soln_error} and by Assertion 1 we could obtain a sequence $\seql{(\alpha_{\T_l},\B_{\T_l})}$, where, at each step $l\in\N$, $(\alpha_{\T_l},\B_{\T_l})\in\T_l$ is determined in \eqref{tgv_soln_error}, such that 
\be
(\alpha_{\T_l},\B_{\T_l})\to (\alpha_\T,\B_\T).
\ee
Also, at each step $l\in\N$, we find the grid $\mathbb G_l(i_l,j_l)$ be such that
\be\label{where_the_true_opt}
(\alpha_\T,\B_\T)\in \mathbb G_l(i_l,j_l)
\ee
where the grid $\mathbb G_l(i_l,j_l)$ is defined in \eqref{eq_finite_grid_one}. Then, in view of \eqref{two_step_estimation}, we have that
\begin{align}\label{grid_estimation_one}
&\max\flp{\CC(\alpha, \B):\,\,(\alpha, \B)\in \mathbb G_l(i_l,j_l)} -\min\flp{\CC(\alpha, \B):\,\,(\alpha, \B)\in \mathbb G_l(i_l,j_l)}\\
&\leq 4KP \fmp{O\fsp{P/l}+1/l}^{1/2}\norm{u_\eta}_{W^{d,1}(Q)}^{1/2}.
\end{align}
We divide into two cases.\\\\
\underline{Case 1:} Assume at step $l$ that $(\alpha_{\T_l},\B_{\T_l})\in \mathbb G_l(i_l,j_l)$. In this case we could directly deduce that
\begin{align}
&\CC(\alpha_{\T_l},\B_{\T_l})-\CC(\alpha_\T, \B_\T)\leq \max\flp{\CC(\alpha, \B):\,\,(\alpha, \B)\in \mathbb G_l(i_l,j_l)} - \CC(\alpha_{\T_l},\B_{\T_l})\\
&\leq \max\flp{\CC(\alpha, \B):\,\,(\alpha, \B)\in \mathbb G_l(i_l,j_l)} - \min\flp{\CC(\alpha, \B):\,\,(\alpha, \B)\in \mathbb G_l(i_l,j_l)};
\end{align}
\underline{Case 2:} Assume at step $l$ that $(\alpha_{\T_l},\B_{\T_l})\notin \mathbb G_l(i_l,j_l)$. In this case, however, in view of the definition of $(\alpha_{\T_l},\B_{\T_l})$, we must have
\be\label{anyway_minimizer}
\max\flp{\CC(\alpha, \B):\,\,(\alpha, \B)\in \mathbb G_l(i_l,j_l)\cap\mathbb T_l}\geq \CC(\alpha_{\T_l},\B_{\T_l}).
\ee
Since if not, $(\alpha_{\T_l},\B_{\T_l})$ would not be a global minimizer among $\mathbb T_l$, which is a contradiction. Therefore, by \eqref{anyway_minimizer} we again have 
\begin{align}
&\CC(\alpha_{\T_l},\B_{\T_l})-\CC(\alpha_\T, \B_\T)\leq\max\flp{\CC(\alpha, \B):\,\,(\alpha, \B)\in \mathbb G_l(i_l,j_l)\cap\mathbb T_l}- \CC(\alpha_\T,\B_\T)\\
&\leq \max\flp{\CC(\alpha, \B):\,\,(\alpha, \B)\in \mathbb G_l(i_l,j_l)} - \min\flp{\CC(\alpha, \B):\,\,(\alpha, \B)\in \mathbb G_l(i_l,j_l)},
\end{align}
where at the last inequality we used the assumption \eqref{where_the_true_opt}. In the end, in view of those two cases discussed above and estimation \eqref{grid_estimation_one}, we observe that 
\begin{align}\label{smooth_result_tv}
&\CC(\alpha_{\T_l},\B_{\T_l})-\CC(\alpha_\T,\B_\T)\\
&\leq \max\flp{\CC(\alpha, \B):\,\,(\alpha, \B)\in\partial \mathbb G_l(i_l,j_l)} -  \min\flp{\CC(\alpha, \B):\,\,(\alpha, \B)\in \mathbb G_l(i_l,j_l)}\\
&\leq 4KP \fmp{O\fsp{P/l}+1/l}^{1/2}\norm{u_\eta}_{W^{d,1}(Q)}^{1/2}
\end{align}
and hence the thesis.\\\\
Now we remove the assumption that $u_\eta\in C^\infty(\bar Q)$. Let $u_\eta^\e\in C^\infty(\bar Q)$ be defined as in Case 2 in the argument used to prove Theorem \ref{main_thm}. Define
\be
u_{\alpha,\B}^\e:=\argmin\flp{\norm{u-u_\eta^\e}_{L^2(Q)}^2+\alpha PV_\B(u):\,\,u\in L^1(Q)}.
\ee
Then by Proposition \ref{prop_approx_fiderly} we have that 
\be
\norm{u_{\alpha,\B}-u_{\alpha,\B}^\e}_{L^2(Q)}\leq \norm{u_\eta^\e-u_\e}_{L^2(Q)},
\ee
for arbitrary $(\alpha,\B)\in \mathbb T$. Then, for any $\delta>0$ be fixed, we could choose $\e>0$ small enough such that 
\be
\norm{u_\eta^\e-u_\e}_{L^2(Q)}<\delta/4\text{ and } \norm{u_\eta^\e}_{W^{d,1}(Q)}\leq \norm{u_\eta}_{L^1(Q)}/\delta^d.
\ee
This, and together with \eqref{smooth_result_tv}, we conclude that 
\begin{align}
\CC(\alpha_{\T_l},\B_{\T_l})-\CC(\alpha_\T,\B_\T)
&\leq 4KP \fmp{O\fsp{P/l}+1/l}^{1/2}\norm{u_\eta}_{W^{d,1}(Q)}^{1/2}+\delta/2\\
&\leq 4KP \fmp{O\fsp{P/l}+1/l}^{1/2}\norm{u_\eta}_{W^{d,1}(Q)}^{1/2}/\delta^d+\delta/2
\end{align}
as desired.
\end{proof}

\subsection{Examples of Training ground}
In this section we give some examples of collection $\Sigma$ that satisfies Assumption \ref{kappa_training_ground}. Recall the structure of operator $\B$ from Notation \ref{operator_B}.
\subsubsection{Operator $\B$ with invertible matrix}
Let $P\in\R^+$ used in Assumption \ref{kappa_training_ground} be given. We define the collection $\Sigma_P$ by
\be\label{sigma_P_define}
\Sigma_P:=\flp{\B:\,\,\abs{(B^h)^{-1}}\leq P,\text{ for each }1\leq h\leq d}.
\ee
We define the $h$-order total variation, say $TV^h$, of $u$ by
\be
PV^d(u) = \abs{H^hu}_{\mb(Q;\mathbb M^{N^h})}.
\ee
where $H^h$ is the $h$-order Hessian operator defined in Notation \ref{operator_B}. We also define the space $BV^d(Q)$ by
\be
BV^d(Q):=\flp{u\in L^1(Q):\,\, TV^d(u)<+\infty},
\ee
with norm
\be
\norm{u}_{BV^d(Q)}:=\norm{u}_{L^1(Q)}+ TV^d(u).
\ee
\begin{proposition}
Let $\B\in \Sigma_P$ be given. Then the space $BV_\B(Q)$ is equivalent to the space $BV^d(Q)$.
\end{proposition}
\begin{proof}
Without lose of generality we assume that $u\in BV^d(Q)\cap C^\infty(Q)$, and in view of the structure of operator $\B$, we have
\begin{align}\label{equivlanet_Pd}
&PV_\B(u)=\sum_{h=1}^d \abs{B^h H^hu}dx\leq \sum_{h=1}^d\abs{B^h}\abs{H^hu}dx\\
&\leq \sum_{h=1}^dTV^h(u)\leq  C\fsp{\norm{u}_{L^1(Q)}+ TV^d(u)},
\end{align}
where at the last inequality we used the Sobolev inequality.\\\\
On the other hand, we have
\be
TV^d(u) = \int_Q\abs{H^du}dx = \int_Q\abs{(B^d)^{-1}B^dH^du}dx \leq \abs{(B^d)^{-1}}_{\ell^\infty}\int_Q\abs{B^d H^du}dx\leq PV_\B(u).
\ee
This, and together with \eqref{equivlanet_Pd}, we are done.
\end{proof}

In the following proposition we show that $\Sigma_P$ satisfies Assumption \ref{kappa_training_ground}, Assertion 2.
\begin{proposition}\label{b1b2_l2different}
Let $\B_1$ and $\B_2\in\Sigma_P$, and $u\in BV^d(Q)$ be given. Then we have
\be\label{use_b1}
\abs{PV_{\B_1}(u)-PV_{\B_2}(u)}\leq \fmp{\sqrt{K}\abs{\B_1- \B_2}_{\ell^\infty}\sum_{h\leq d}\abs{(B^h_1)^{-1}}_{\ell^\infty}}PV_{\B_1}(u),
\ee
and
\be\label{use_b2}
\abs{PV_{\B_1}(u)-PV_{\B_2}(u)}\leq \fmp{\sqrt{K}\abs{\B_1- \B_2}_{\ell^\infty}\sum_{h\leq d}\abs{(B^h_2)^{-1}}_{\ell^\infty}}PV_{\B_2}(u),
\ee
where $K$ is defined in \eqref{A_quasiconvexity_operator}.
\end{proposition}
\begin{proof}
Assume for a moment that $u\in C^\infty(\bar Q)$, we compute that 
\be
\abs{PV_{\B_1}(u)-PV_{\B_2}(u) }= \abs{\int_Q\abs{\B_1u}dx-\int_Q\abs{\B_2 u}dx}\leq \int_Q\abs{\B_1u-\B_2u}dx.
\ee
Next, we observe that, for $x\in Q$,
\be
(\B_1-\B_2)u(x) =\sum_{h\leq d} (B^h_1- B^h_2)H^h u (x)= \sum_{h\leq d} (B^h_1- B^h_2) (B^h_1)^{-1} B^h_1 H^h u(x).
\ee
Thus, we have
\begin{align}
\int_Q\abs{(\B_1-\B_2)u}dx&\leq \fmp{\sum_{h\leq d}\abs{(B^h_1- B^h_2)}_{\ell^\infty}\abs{ (B^h_1)^{-1}}_{\ell^\infty}  }\sum_{h\leq d}\abs{B^h_1 H^h u}_{\mb(Q;\R^{N^h})} \\
&\leq \sqrt{K}\fmp{\sum_{h\leq d}\abs{(B^h_1- B^h_2)}_{\ell^\infty}\abs{ (B^h_1)^{-1}}_{\ell^\infty}}\int_Q\abs{\B_1u}dx.
\end{align}
Thus, we have that 
\begin{align}\label{smooth_dejin}
&\abs{PV_{\B_1}(u)-PV_{\B_2}(u)} \leq \int_Q\abs{\B_1u-\B_2u}dx\\
&\leq \fmp{\sqrt{K}\abs{\B_1- \B_2}_{\ell^\infty}\sum_{h\leq d}\abs{(B^h_1)^{-1}}_{\ell^\infty}}PV_{\B_1}(u).
\end{align}
To conclude, we use an approximation sequence $\seqe{u_\e}\subset C^\infty(\bar Q)$ from Corollary \ref{smooth_strict_approximation_finite} such that $u_\e\to u$ in $L^1$ and
\be
PV_{\B_1}(u_\e)\to PV_{\B_1}(u)\text{ and }PV_{\B_2}(u_\e)\to PV_{\B_2}(u).
\ee
This, and together with \eqref{smooth_dejin}, we obtain \eqref{use_b1} as desired. Lastly, we remark that we could conclude \eqref{use_b2} in the same way and hence the thesis.
\end{proof}
\textbf{Remark.} By \eqref{sigma_P_define}, \eqref{use_b1}, and \eqref{use_b2}, we conclude that 
\be
\abs{PV_{\B_1}(u)-PV_{\B_2}(u)}\leq \fmp{d\sqrt{K}P\abs{\B_1- \B_2}_{\ell^\infty}}PV_{\B_1}(u),
\ee
and
\be
\abs{PV_{\B_1}(u)-PV_{\B_2}(u)}\leq \fmp{d\sqrt{K}P\abs{\B_1- \B_2}_{\ell^\infty}}PV_{\B_2}(u).
\ee
Thus, by setting 
\be
O\fsp{\abs{\B_1- \B_2}_{\ell^\infty}}={d\sqrt{K}P\abs{\B_1- \B_2}_{\ell^\infty}},
\ee
we conclude that $\Sigma_P$ satisfies Assumption \ref{kappa_training_ground}.
\subsubsection{Operators with Energy constraint} We briefly mention that in our previous work \cite{liu2016Image}, a collection of first order operators $\B$ is introduced, based on some natural quasi-convex constraint. The precise definition is pretty complicated so we decide not to report it again here but refer our readers to \cite[Section 6]{liu2016Image} for future reference.

\section{Experimental insights, further extensions, and upcoming works}\label{fsaqs_sec}
\subsection{Numerical simulations}\label{primal_dual_rmk}
We remark that the reconstructed image $u_{\alpha,\B}$ defined in \eqref{ABsolution_map}, for any given $(\alpha,\B)\in\mathbb T$, can be computed by using the primal-dual algorithm presented in \cite{chambolle2011first}. Indeed, we could recast the minimizing problem \eqref{ABsolution_map} as the min-max problem 
\be
\min\flp{\max\flp{\norm{u-u_\eta}_{L^2}^2+\alpha \fjp{u,\B^\ast\vp}:\,\,\vp\in C_c^\infty(Q;\rk)}:\,\, u\in L^1(Q)},
\ee
and then the primal-dual method presented in \cite{chambolle2011first} applied.\\\\
Next, we present how we put Theorem \ref{PV_finiteapprox_result} into practical use. \\\\
\noindent\fbox{%
    \parbox{\textwidth}{%
Let $u_c\in L^2(Q)$ and $u_\eta\in L^2(Q)$ be given. Let an acceptable error $\e>0$ be given.
\begin{itemize}
\item
Initialization: Choose an acceptable error $\e>0$. Choose the box-constraint constant $P>0$.
\item
Step 1: Let $\delta=\e/2$ and increase step $l\in\N$ until the training error given in Assertion 2, Theorem \ref{PV_finiteapprox_result}, less or equal to $\e/2$.
\item
Step 2: Determine one global minimizer $(\alpha_{\T_l},\B_{\T_l})$ of assessment function $\CC(\alpha,\B)$  over the finite training ground $\T_l$. Then, by Theorem \ref{PV_finiteapprox_result} we have that
\be
\abs{\CC(\alpha_{\T_l},\B_{\T_l})-\CC(\alpha_\T,\B_\T)}\leq \e,
\ee
\item
Step 3: the reconstructed image $u_{\alpha_{\T_l},\B_{\T_l}}$ is then a desired optimal reconstructed result within the acceptable error range.
\end{itemize}
}%
}\\

For the sake of appropriate comparison, we apply our proposed training scheme $\mathcal T$ (\eqref{ABtraining_0_1}-\eqref{ABsolution_map}) on the image given in Figure \ref{fig:clean_noise} with the following training grounds
\be\label{test_ground0}
\T^0:=\fmp{0,1}\times \flp{\B_0},\text{ where }\B_{0}:=[1,0;0,1]
\ee
\be\label{test_ground1}
\T^1:=\fmp{0,1}\times \flp{\B_{s}:\,\,-0.5\leq s\leq 0.5},\text{ where }\B_{s}:=[1,s;0,1]
\ee
\be\label{test_ground2}
\T^2:=\fmp{0,1}\times \flp{\B_{s,t}:\,\,-0.5\leq s,t\leq 0.5},\text{ where }\B_{s,t}:=[1,s;t,1],
\ee
where we use super-script to avoid confusion with the finite training ground $\T_l$. Note that the training ground $\T^0$ gives the original training scheme $\mathcal B$ (\eqref{scheme_B1_BV}-\eqref{scheme_B2_BV}) with $TV$ regularizer only. We perform numerical simulations of the images shown in Figure \ref{fig:clean_noise}: the first image represents a clean image $u_c$, whereas the second one is a noised version $u_\eta$. We summarize our simulation results in Table \ref{table_test_result} below.
\begin{table}[!h]
\begin{tabular}{|l|l|l|l|l|l|}
\hline
Training ground & optimal solution & minimum assessment value   \\ \hline
$\T^0$ & $\alpha_{\T^0}=0.048$  & 14.8575    \\ \hline
$\T^1$ &$\alpha_{\T^1}=0.052$, $s_{\T^1}=0.4$  &12.8382  \\ \hline
 $\T^2$&$\alpha_{\T^2}=0.052$, $s_{\T^2}=-0.2$, $t_{\T^2}=0.5$ & 12.2369   \\ \hline
\end{tabular}
\caption{minimum assessment value for scheme $\mathcal T$ over training ground defined in \eqref{test_ground0}, \eqref{test_ground1}, and \eqref{test_ground2}}
\label{table_test_result}
\end{table}
We observe that, from Table \ref{table_test_result}, as the training ground expand, the minimum value of assessment function $\CC(\alpha,\B)$ decreased. In another word, our new regularizer $PV_\B$ indeed provides an improved reconstructed result compare with $TV$ regularizer. However, we remark that the extension of training ground results in a increasing of considerable large amount of CPU time, this would not be a big problem for practical application since we only need to use it one time for a given data set, and the structure of finite training ground $\T_l$ allows us to use parallel computing very efficiently and hence reduce the CPU usage.

\begin{figure}[!h]
  \centering
        \includegraphics[width=1.0\linewidth]{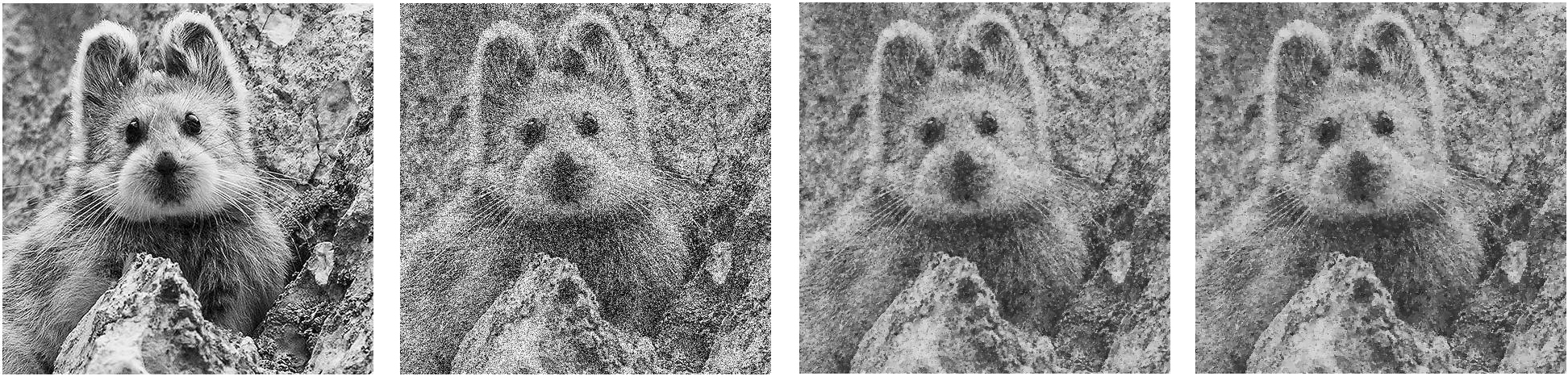}
\caption{From left to right: clean image $u_c$; corrupted image $u_\eta$ (with heavy artificial Gaussian noise); the optimally reconstructed image at ${\alpha_{\T^0}}$, the optimal reconstructed image at $\fsp{\alpha_{\T^2},\B_{\T^2}}$}
\label{fig:clean_noise}
\end{figure}
To explore the numerical landscapes of the assessment function $\mathcal A(\alpha,\B )$ with respect to $\B$, we consider the following training ground with intensity parameter $\alpha$ fixed
\be
\T=\flp{0.025}\times \flp{\B_{s,t}:\,\,-0.5\leq s,t\leq 0.5}.
\ee
and we plot in Figure \ref{fig:mesh_contour_0_1_skew} the mesh and contour images.

\begin{figure}[!h]
\begin{subfigure}{.495\textwidth}
  \centering
        \includegraphics[width=1.0\linewidth]{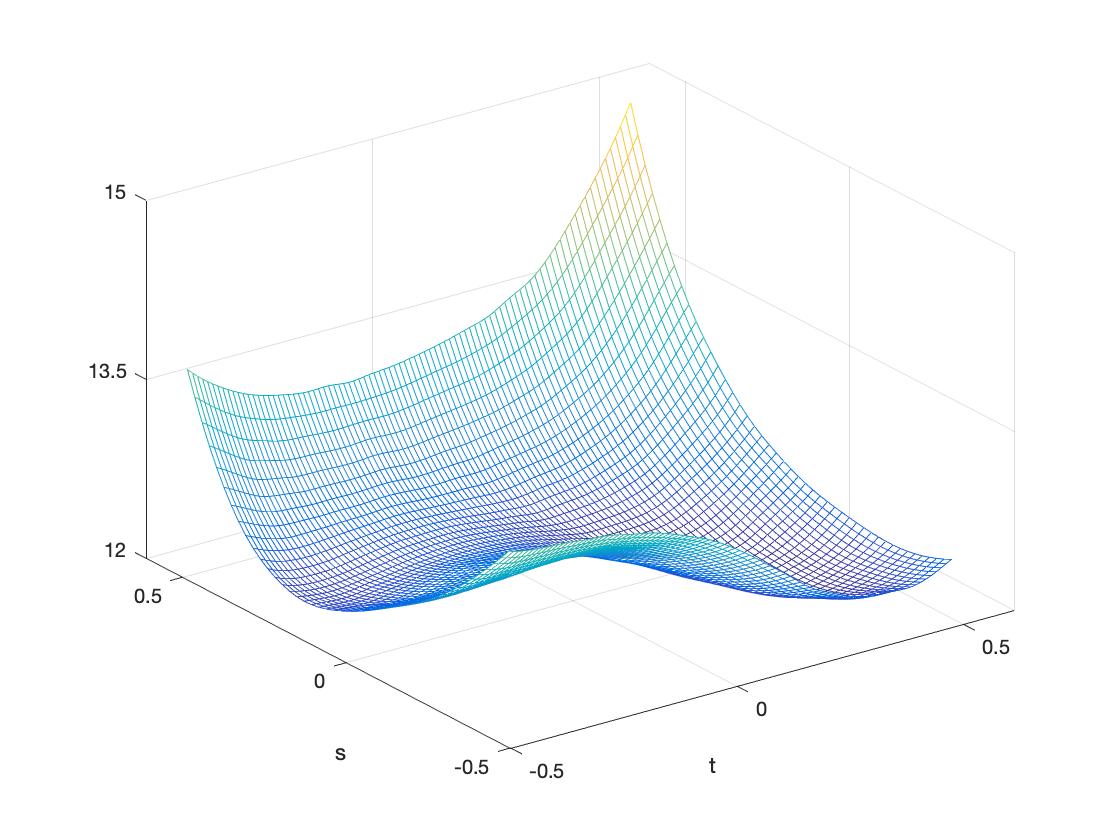}
\end{subfigure}
\begin{subfigure}{.495\textwidth}
  \centering
        \includegraphics[width=1.0\linewidth]{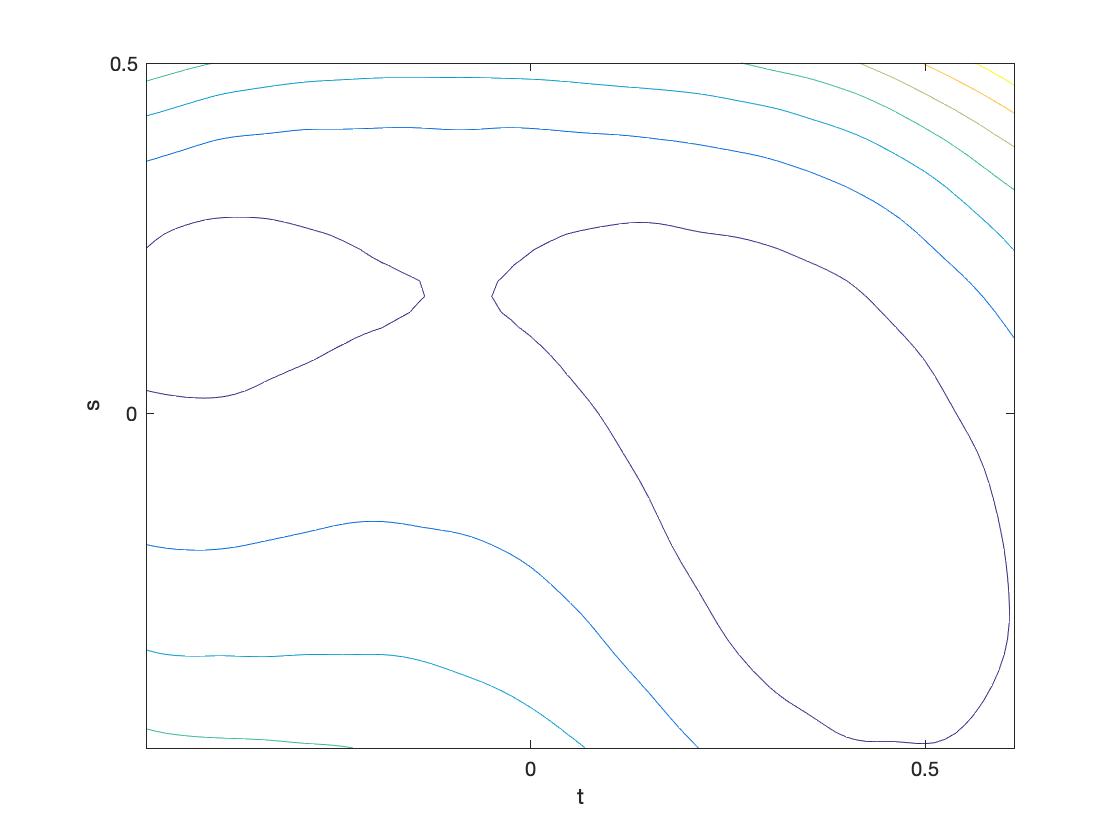}
\end{subfigure}
\caption{From left to right: mesh and contour plot of the assessment function $\mathcal A(0.052,s,t)$ in which $(s,t)\in[-0.5,0.5]^2$. We see from above figures that indeed $\CC(\cdot)$ is not convex.}
\label{fig:mesh_contour_0_1_skew}
\end{figure}
We remark that the introduction of $PV_\B$ regularizers into the training scheme is only meant to expand the training choices, but not to provide a superior semi-norm with respect to the standard $TV$ semi-norm. The fact whether the optimal regularizer is $TV$ or another intermediate regularizer, is completely dependent on the given training image $u_\eta=u_c+\eta$. Moreover, we remark that the results discovered in this article are not restricted to the imaging processing problems only. It can be generally applied to parameter estimation problems of variational inequalities, as long as a suitable assessment function can be defined.

\subsection{Further generalization of $PV_\B$ regularizer}\label{sec_further_generaliza}
\subsubsection{Extension with variating underlying Euclidean norm}
For $x=(x_1,\ldots,x_K)\in\rk$, we recall that, for $p\in[1,+\infty)$, the $p$-Euclidean norm of $x$ is defined as
\be
\abs{x}_p = \fsp{\sum_{i=1}^K \abs{x_i}^p}^{1/p}\text{ and }\abs{x}_\infty=\max\flp{\abs{x_i}:\,\,i=1,\ldots,K}.
\ee 
Note that for $p=2$, we recovery the standard Euclidean norm $\abs{x}=\abs{x}_2$, which is used in \eqref{BVB_norm}.\\\\
In the spirit of \cite{carolapani2018bilevel}, we could generalize regularizer $PV_\B$ by variating the underlying Euclidean norm. To be precise, we define
\be
PV_{p,\B}:=\sup\flp{\int_Q u\,\B^\ast\vp\,dx:\,\,\vp\in C_c^\infty(Q;\rk),\,\,\abs{\vp(x)}_{p}^\ast\leq 1},
\ee
where $\abs{\cdot}_p^\ast$ represents the dual norm of $\abs{\cdot}_p$, as well as a new training scheme
\begin{flalign}
{\text{Level 1. }}&\,\,\,\,\,\,(\alpha_\T,p_T,\B_\T)\in \argmin\flp{\norm{u_c-u_{\alpha,\B}}_{L^2(Q)}^2:\,\,(\alpha,p,\B)\in\mathbb T},\tag{$\mathcal T$-L1}\label{ABtraining_0_1p}&\\
{\text{Level 2. }}&\,\,\,\,\,\,u_{\alpha,p,\B}:=\argmin\flp{\norm{u-u_\eta}_{L^2}^2+\alpha PV_{p,\B}(u),\,\, u\in L^1(Q)},\tag{$\mathcal T$-L2}\label{ABsolution_mapp}&
\end{flalign}
with the training ground
\be
\T:=\bar\R\times [1,+\infty]\times\Sigma.
\ee 
We remark that, both Theorem \ref{main_thm} and Theorem \ref{PV_finiteapprox_result} holds on this new training scheme with training ground $\T$ and $\T = [0,\kappa]\times[1,+\infty]\times\Sigma$, respectively. The prove is identical to the argument presented before and the argument used in \cite{carolapani2018bilevel} when deal with parameter $p$, so we decide to not report it here to avoid redundancy.

\subsubsection{Extension with real order derivative} 
Let $s=(s_1,s_2\ldots, s_d)\in[0,1)^d$ be given, we define the real $s$-order operator $\B[s]:L^1(Q)\to \mathcal D'(Q;\rk)$ by
\be\label{A_quasiconvexity_operator}
\B[s] u:= \sum_{h=1}^dB^h (H^{h-s_h}u)\quad\text{for every }\,u\in L^1(Q),
\ee
where $H^{h-s_h}$ represent the $h-s_h$ order Hessian of $u$. For example, for $h=1$ and $s_1\in[0,1)$, we have 
\be
H^{1-s_1}u = [\partial_1^{1-s_1}u,\partial_2^{1-s_1}u],
\ee
where $\partial_1^{1-s_1}u$ represent the \emph{Riemann-Liouville}  fractional partial derivative with order $1-s_1\notin\N$ (see \cite[Definition 2.6]{2018arXiv180506761L}). Then we define
\be
PV_{s,p,\B}:=\sup\flp{\int_Q u\,(\B^\ast[s])\vp\,dx:\,\,\vp\in C_c^\infty(Q;\rk),\,\,\abs{\vp(x)}_{p}^\ast\leq 1}.
\ee
and the new scheme 
\begin{flalign}
{\text{Level 1. }}&\,\,\,\,\,\,(\alpha_\T,p_\T,s_\T,\B_\T)\in \argmin\flp{\norm{u_c-u_{\alpha,p,s,\B}}_{L^2(Q)}^2:\,\,(\alpha,p,s,\B)\in\mathbb T},\tag{$\mathcal T$-L1}\label{ABtraining_0_1ps}&\\
{\text{Level 2. }}&\,\,\,\,\,\,u_{\alpha,p,s,\B}:=\argmin\flp{\norm{u-u_\eta}_{L^2}^2+\alpha PV_{s,p,\B}(u),\,\, u\in L^1(Q)}.\tag{$\mathcal T$-L2}\label{ABsolution_mapps}&
\end{flalign}
with training groud
\be
\T:=\operatorname{cl}\fsp{\R^+}\times [1,+\infty]\times[0,1)^d\times\Sigma.
\ee
We remark that Theorem \ref{main_thm} holds, with additional technics needed when deal with order parameter $s$ which we report separately in \cite{2019arXiv180506761L}. However, dual to the complexity of fractional order derivative, we can not directly deduce an analogously version of Theorem \ref{PV_finiteapprox_result}.

\subsection{Upcoming works}
In \cite{liu2016Image}, a \emph{PDE}-constraint total generalized variation, say $PGV^2_{\mathscr D}$, in multi-dimensions $N\in\N$ is defined as follows
\be
{PGV_{\mathscr D}^2(u)}:=\inf\flp{\abs{\nabla u- v_0}_{\mathcal{M}_b(Q)}+\abs{\mathscr D v_0}_{\mb(Q;\mathbb M^{N\times N})}:\,v_0\in L^1(Q;\mathbb M^{N\times N})},
\ee
where $\mathscr D$: $L^1(Q;\rn)\to\mathcal D'(Q;\mathbb M^{N\times N})$ is a first order differential operator with some natural \emph{PDE} constraint (see Section 3 in \cite{liu2016Image}).\\\\
In our follow-up work, we propose to construct an unified approach to regularizers $PGV_{\mathscr D}^2$ and $PV_{s,\B}$ via
\be
PGV_{s,\B,t,\mathscr D}(u):=\inf\flp{\abs{\B[s] u- tv_0}_{\mathcal{M}_b(Q)}+t\abs{\mathscr D[t] v_0}_{\mb(Q;\mathbb M^{N\times N})}:\,v_0\in L^1(Q;\mathbb M^{N\times N})}.
\ee
We shall equip the training scheme $\mathcal T$ with this new regularizer and provide an analogously version of Theorem \ref{main_thm} and Theorem \ref{PV_finiteapprox_result} in our follow-up work.\\\\
\textbf{Acknowledgments.} PL acknowledges support from the EPSRC Centre Nr. EP/N014588/1 and the Leverhulme Trust project on Breaking the non-convexity barrier. The author also gratefully acknowledge the support of NVIDIA Corporation with the donation of a Quadro P6000 GPU used for this research.

\bibliographystyle{abbrv}
\bibliography{PBR_A}{}

\end{document}